\documentclass[10pt,preprint,english]{elsarticle}

\usepackage[english]{babel}
\usepackage[T1]{fontenc}
\usepackage[utf8]{inputenc}

\usepackage[left=2cm,right=2cm,top=2cm,bottom=2cm]{geometry}

\usepackage{lipsum}
\usepackage{setspace}

\usepackage{amsfonts}
\usepackage{amsthm}
\usepackage{latexsym}
\usepackage{hyperref}
\usepackage{amssymb, amsmath}
\usepackage{enumitem}

\usepackage{pdfsync}
\usepackage{epstopdf}
\usepackage{hyperref}

\numberwithin{equation}{section}

\usepackage{xcolor}
\usepackage{color}


\def\R{\mathbb{R}}

\def\R{{\mathbb R}}

\newtheorem{defi}{Definition}[section]
\newtheorem{theorem}{Theorem}[section]
\newtheorem{prop}{Proposition}[section]

\newtheorem{lemma}{Lemma}[section]
\newtheorem{remark}{Remark}[section]

\newcommand{\red}{\textcolor{black}}

\newcommand{\QEDA}{\null\nobreak\hfill\ensuremath{\blacksquare}}%
\newcommand{\QEDB}{\null\nobreak\hfill\ensuremath{\blacksquare}}%

\begin{document}

\begin{frontmatter}

\title{Multiplicity results for a subcritical Hamiltonian system with concave-convex nonlinearities}

\author[KMA]{Oscar Agudelo\corref{cor1}}
\ead{oiagudel@kma.zcu.cz}
\cortext[cor1]{Corresponding author}

\author[UNIMI]{Bernhard Ruf}
\ead{bernhard.ruf@unimi.it}

\author[UNAL]{Carlos V\'{e}lez}
\ead{cauvelez@unal.edu.co}

\address[KMA]{Department of Mathematics and NTIS, Faculty of Applied Sciences, University of West Bohemia in Pilsen
\\
Univerzitn\'{i} 22, 306 14 Plze\v{n}, Czech Republic.}

\address[UNIMI]{Dipartimento di Matematica, Universit\'{a} di Milano
via C. Saldini 50, Milano 20133, Italy.}

\address[UNAL]{Departamento de Matem\'{a}ticas, Universidad Nacional de Colombia Sede Medell\'{i}n, Medell\'{i}n, Colombia.}

\numberwithin{equation}{section}

\begin{abstract}
We study  the {\it Hamiltonian elliptic system}  
\begin{eqnarray}\label{HS1-abstract} 
\left\{
\begin{aligned}
-\Delta u  & = \lambda |v|^{r-1}v +|v|^{p-1}v \qquad &\hbox{in} \ \ \Omega ,\\
-\Delta v   & = \mu |u|^{s-1}u +|u|^{q-1}u \qquad &\hbox{in} \ \ \Omega ,\\
u &>0, \ v>0 \qquad \, &\hbox{in} \ \ \Omega ,\\
u &=v = 0 \qquad \quad &\hbox{on} \quad \partial \Omega,
\end{aligned}
\right.
\end{eqnarray}
where $\Omega \subset \mathbb {R}^N$ is a smooth bounded domain,  $\lambda$ and $ \mu $ are nonnegative parameters and $r,s,p,q>0$.
Our study includes the case in which the nonlinearities in \eqref{HS1-abstract} are concave near the origin and convex near infinity, and we focus on the region of non-negative {\it pairs of parameters} \red{$(\lambda,\mu)$} that guarantee exis\-tence and multiplicity of solutions of \eqref{HS1-abstract}. \red{In particular, we show the existence of a strictly decreasing curve $\lambda_*(\mu)$ on an interval $[0, \mu]$ with $\lambda_*(0)> 0, \lambda_*(\mu) = 0$ and such that the system has two solutions for $(\lambda,\mu)$ below the curve, one solution for $(\lambda, \mu)$ on the curve and no solution for $(\lambda, \mu)$ above the curve. A similar statement holds reversing $\lambda$ and $\mu$.} This work is motivated by some of the results from \cite{AMBROSSETIBREZISCERAMI1994},
 \cite{dosSantos2009-2} and \cite{AGUDELORUFVELEZ2023}. 
\end{abstract}

\medskip



\end{frontmatter}

\section{Introduction}

Let us consider the elliptic system 
\begin{eqnarray}\label{HS1} 
\left\{
\begin{aligned}
-\Delta u  & = \lambda |v|^{r-1}v +|v|^{p-1}v \qquad &\hbox{in} \ \ \Omega ,\\
-\Delta v   & = \mu |u|^{s-1}u +|u|^{q-1}u \qquad &\hbox{in} \ \ \Omega ,\\
u &>0, \ v>0 \qquad \, &\hbox{in} \ \ \Omega ,\\
u &=v = 0 \qquad \quad &\hbox{on} \quad \partial \Omega,
\end{aligned}
\right.
\end{eqnarray}
{where} $\Omega \subset \mathbb {R}^N$ is a smooth bounded domain,  $N \geq 2$, $\Delta$ is the Laplace operator,  $\lambda$ and $ \mu $ are nonnegative parameters and $r,s,p,q\in(0,\infty)$.

\vskip 6pt
A particular case of our result concerns the situation where both nonlinearities
are concave near zero and convex for large values, that is
\begin{equation}\tag{A0}\label{hyp:concconv}
0 < s, r < 1 \quad \hbox{and}\quad 1 < p, q\quad \hbox{with}\quad  \frac{1}{p + 1} + \frac{1}{
q + 1} > 1 -\frac{2}{N}.    
\end{equation}

The second condition says that the asymptotic growth of the nonlinearities
is superlinear and subcritical (with respect to the {\it ”critical hyperbola”}). If
these exponents are fixed, then the solvability of \eqref{HS1} depends on the parameters $\lambda$ and $\mu$. Indeed, we show

\begin{theorem}\label{theo:Esp-caseMainTheo}
Assume \eqref{hyp:concconv} holds true and that $qr<1$ and $ps<1$. {There exists a continuous and strictly decreasing function $\lambda_*: [0, \bar{\mu}] \to [0, a]$ with
$\lambda_*(0)= a > 0$, and $\lambda_*(\bar{\mu})=0$} such that:
\begin{itemize}
    \item if $\lambda\in (0, \lambda_*(\mu))$, then the system \eqref{HS1}has at least two positive solutions;
    \item if $\lambda=\lambda_*(\mu)$, then the system has at least one solution;
\item if $\lambda>\lambda_*(\mu)$, then the system has no solution.
\end{itemize}

A similar statement holds reversing $\lambda$ and $\mu$.
\end{theorem}

The approach to this problem is as follows: we invert the (odd and monotone increasing) nonlinearity in the second equation, and obtain $u$ in terms of $\Delta v$,
that is $-\psi(\mu,\Delta v)$. Inserting this into the first equation in \eqref{HS1} we obtain a fourth order quasilinear operator, which behaves as a superlinear $\frac{1}{s}-$bi-Laplacian
near zero and like a sublinear $\frac{1}{q}$-bi-Laplacian near infinity  (see \eqref{4thorderBVP} on page \pageref{pgref:psi}). The Dirichlet boundary conditions in \eqref{HS1} are transformed into Navier boundary conditions for $v$ and hence we consider this
operator in a closed subspace of $W^{2, \frac{q+1}{q}}(\Omega)$. The equation we then study is
$$
\Delta \psi (\mu,\Delta v)=\lambda |v|^{r-1}v + |v|^{p-1}v.
$$ 

We show various properties and regularity results for such equations, which will be of interest also to other problems with such a setting. The existence and multiplicity results are obtained via variational methods for the energy functional associated to this fourth order quasilinear equation. We now give a detailed review of some of the known results for such systems, and state our result in more generality (see Theorem \ref{theo:bifur_diagram} below).

\vskip 6pt

 If $\mu=\lambda$, $r=s$ and $p=q$, one may look for solutions $(u,v)$ of \eqref{HS1} with $u=v$ in $\Omega$. In this case, system \eqref{HS1} reduces to the scalar boundary value problem ({BVP}) \begin{equation}\label{eq:Ambr-Brezis-Cerami}
-\Delta v = \lambda |v|^{r-1}v + |v|^{p-1}v, \quad  v>0 \quad \hbox{in} \quad \Omega, \qquad v=0 \quad \hbox{on} \quad \partial \Omega. 
\end{equation} 

 In \cite{BREZISNIRENBERG1983}, Brezis and Nirenberg  studied the threshold values of $\lambda$ that separate exis\-tence and non-existence of solutions of the BVP \eqref{eq:Ambr-Brezis-Cerami} in the case $r=1$, $N\geq 3$ and $p=\frac{N+2}{N-2}$.   Ambrosetti, Brezis and Cerami in \cite{AMBROSSETIBREZISCERAMI1994} did the corres\-ponding analysis for \eqref{eq:Ambr-Brezis-Cerami} in the case $r\in(0,1)$, $N\geq 3$ and $1<p \leq \frac{N+2}{N-2}$, that is, when the nonlinearity is concave near the origin and convex for large values of $v$. 

\vskip 6pt

 In the case $r=s=1$ and $p\neq q$, Van der Vorst in \cite{Vandervorst1992} and Mitidieri in \cite{Mitidieri1993} used generalized versions of the Poho\-\v{z}aev identity (see also \cite{Pohozaev1965}) to prove that in star-shaped domains existence of solutions of \eqref{HS1} is only possible if
\begin{equation}\label{cond:crit_hyperbola}
\frac{1}{p+1}+\frac{1}{q+1}\geq 1 -\frac{2}{N}.
\end{equation}  

\vskip 6pt

  In what follows $\lambda_1>0$ denotes the principal eigenvalue of $-\Delta$ in $H_0^1(\Omega)$. 
Assuming $r=s=1$, 
\begin{equation}\label{eq:subcrit-hyperb}
\frac{1}{p+1} + \frac{1}{q+1} > 1 - \frac{2}{N},
\end{equation}
 and that $\mu \lambda <\lambda_1^2$, the authors in \cite{HulshofVanderVorst1993} proved existence of at least one smooth solution for a slightly more general version of system \eqref{HS1}. In \cite{Vandervorst1992}, the author shows non-existence of non-negative solutions of \eqref{HS1} when $\mu \lambda \geq \lambda_1^2$. 

\vskip 6pt

 Using a dual formulation due to Clarke and Ekeland (see \cite{ClarkeEkeland1980}) and still assuming $r=s=1$, the authors in \cite{HulshofMitidieriVanderVorst1998} study the case when equality in \eqref{cond:crit_hyperbola} {holds}. In \cite{HulshofMitidieriVanderVorst1998}, precise asymptotic information about the {\it ground state solutions} of \eqref{HS1} when $\Omega=\R^N$, $\mu=\lambda=0$ is required and taken from the analysis made by Hulshof and Van der Vorst in \cite{HulshofVanderVorst1996}.\\ 

Inspired by these results, recently Guimar\~{a}es and Moreira dos Santos, in \cite{Guimaraes-dosSantos2023}, study system \eqref{HS1} also in the critical case (i.e.{,} assuming equality holds in \eqref{cond:crit_hyperbola}) when $rs\geq 1$, $0<r<p$ and $0<s<q$. The authors raise interesting questions, \red{especially} with regard to the case $N=3$, and prove that if either a) $N\geq 4$ or b) $N= 3$ and $p\leq 7/2$ or $p\geq 8$, then 
\begin{itemize}
    \item there exists  a positive classical solution for arbitrary pairs of parameters $\lambda>0$ and $\mu>0$ when $rs >1$,
    \item there exists a positive classical solution when $\lambda \mu ^r$ is suitable small and $rs =1$.
\end{itemize}

\vskip 6pt

We refer the reader to \cite{dosSantos2009} for a generalization of some the aforementioned results to the case of cooperative systems.

\vskip 3pt
Moreira dos Santos, in {\cite{dosSantos2009-2}}, considers systems of the type \eqref{HS1} when $\mu=0$ and $p>\frac{1}{q}>r>0$ (according to our notations). Explicit values of $\lambda$ for existence and non-existence of nonnegative solutions of the system \eqref{HS1} are studied assuming, in addition, inequality \eqref{cond:crit_hyperbola}. The case of equality in \eqref{cond:crit_hyperbola} is treated with an additional restriction on the dimension $N$ (see Theorem 1.3 and Theorem 1.4 in \cite{dosSantos2009-2}).  We observe that when $q>1$, the  condition $p>\frac{1}{q}>r>0$ implies $0<r<1$. So, the results in \cite{dosSantos2009-2} extend to one-parameter Hamiltonian elliptic systems the results from \cite{AMBROSSETIBREZISCERAMI1994}. It is these two references that motivated of our interest in (and inspired our study of) \eqref{HS1}.

\vskip 3pt

Assuming that $0<r,s<1$ and $p,q>1$, but not \eqref{cond:crit_hyperbola}, the authors in \cite{AGUDELORUFVELEZ2023} proved the existence of one smooth solution to \eqref{HS1} when $\lambda, \mu\geq 0$ are sufficiently small. It is also proven in Theorem 1.2 in \cite{AGUDELORUFVELEZ2023} that smooth solutions to \eqref{HS1} do not exist provided $\mu,\lambda\geq 0$ satisfy the inequality
\begin{equation}\label{ineq:mulambdahyperb}
\mu^{\frac{q-1}{q-s}}\lambda^{\frac{p-1}{p-r}}> C_{rspq}\lambda_1^2,
\end{equation}
where $C_{rspq}:= \left(C_{rp}C_{sq}\right)^{-1}$ with 
\begin{small}
$$
\begin{aligned}
C_{rp}:=&\Big[\Big(\frac{1-r}{p-1}\Big)^{-\frac{1-r}{p-r}} + \Big(\frac{1-r}{p-1}\Big)^{\frac{p-1}{p-r}}\Big],\\
C_{sq}:=&\Big[\Big(\frac{1-s}{q-1}\Big)^{-\frac{1-s}{q-s}} + \Big(\frac{1-s}{q-1}\Big)^{\frac{q-1}{q-s}}\Big].
\end{aligned}
$$
\end{small}

{We remark that 
$$
\lim \limits_{r\to 1^-}C_{rp}=1 \quad \hbox{and} \quad \lim \limits_{s\to 1^-}C_{sq}=1
$$ and hence when taking the limit as $r,s\to 1^-$ in \eqref{ineq:mulambdahyperb}, we obtain $\mu\lambda \geq \lambda_1^2$, which is the sufficient condition given in \cite{Vandervorst1992} for non-existence of nonnegative solutions to \eqref{HS1} for the case $r=s=1$. In this sense, Theorem 1.2 in \cite{AGUDELORUFVELEZ2023} extends and complements the aforementioned result in \cite{Vandervorst1992}.}

\vskip 6pt


In this work we continue the study started in \cite{AGUDELORUFVELEZ2023} and  we focus on  the {\it region of pairs} of non-negative parameters $(\mu,\lambda)$ that guarantee exis\-tence and multiplicity of solutions \eqref{HS1} when $0<r,s<1$. Consider the following hypotheses:

\begin{itemize}
\item[{\rm (A1)}] (subcriticality)
\begin{equation}\label{eq:crit-hyperb}
\frac{1}{p+1} + \frac{1}{q+1} > 1 - \frac{2}{N}.
\end{equation}

\item[{\rm (A2)}] $0<r<\min (1,p)$ and $s{<}\min(1,q)$;

\item[{\rm (A3)}] $r\in (0, \frac{1}{q})$.

\end{itemize}

In what follows, solutions of \eqref{HS1} are understood to be $C^2(\Omega)\cap C(\overline{\Omega})$-classical solutions unless otherwise stated. Our main result reads as follows.

\begin{theorem}
\label{theo:bifur_diagram}
Assume {\rm (A1)}, {\rm (A2)} and {\rm (A3)} and $qp>1$. Then, the region 
$$
S:=\Big\{(\lambda,\mu)\in (0,\infty)\times [0,\infty)\,:\, \eqref{HS1}\quad \hbox{has a solution} \Big\}
$$
is non-empty and path connected. Even more, there exists a {continuous non-increasing} function $\lambda_*:[0,\infty)\to [0,\infty]$ such that for any $\mu \geq 0$,
\begin{itemize}
\item[i.] if $\lambda_*(\mu)>0$ and $\lambda\in (0,\lambda_*(\mu))$, the system \eqref{HS1} has at least two solutions for the pair $(\lambda,\mu)$; 

\item[ii.] if 
  $ 0<\lambda_*(\mu)<\infty$, for $\lambda=\lambda_*(\mu)$ the system \eqref{HS1} has at least one solution for the pair $(\lambda,\mu)$;

\item[iii.] if $\lambda_*(\mu)<\infty$ and $\lambda \in (\lambda_*(\mu),\infty)$, the system \eqref{HS1} has no solution for the pair $(\lambda,\mu)$; 

{\item[iv.] $\lambda_*$ is strictly decreasing in the set $\{\mu \in (0,\infty)\,/\, \lambda_*(\mu)\in (0,\infty)\}$;}

\item[{v.}] if $p,q>1$, for each $\mu\geq 0$, $\lambda_*(\mu)<\infty$ and $\mu^{\frac{q-1}{q-s}}\lambda_*(\mu)^{\frac{p-1}{p-r}}\leq {C_{rspq}}\lambda_1^2 $.

\end{itemize}
\end{theorem}

\begin{remark}
Existence of solutions to \eqref{HS1} in the case $\mu=0$ and $\lambda=0$ was already proven by Hulshof and Van der Vorst in Theorem 1 from \cite{HulshofVanderVorst1993}. We remark that Theorem 1 in \cite{HulshofVanderVorst1993} can not be summoned to obtain solutions of \eqref{HS1} in the case when either $\mu>0$ or  $\lambda>0$. 
\end{remark}

\begin{remark}
As mentioned above, the existence of solutions to \eqref{HS1} is studied in
\cite{dosSantos2009-2}  when $\mu =0$ and $p>\frac{1}{q}>r>0$ (according to our notations). By switching the roles of $\mu$ and $\lambda$, $p$ and $q$, $r$ and $s$, respectively, the results in \cite{dosSantos2009-2} also apply to problem \eqref{HS1} when $\lambda =0$ and  $pq>1>sp>0$. This leaves open the question of existence when \underline{both} parameters $\lambda$ and $\mu$ are positive, which is the main theme in Theorem \ref{theo:bifur_diagram}. In this direction, Theorem \ref{theo:bifur_diagram} generalizes the results from \cite{dosSantos2009-2}.
\end{remark}

\begin{remark}
{One essential step in the proof of part {\it i.} in Theorem \ref{theo:bifur_diagram} consists in proving existence of at least two nontrivial solutions to \eqref{HS1+} when $\mu$ and $\lambda$ are both positive and small (see Propositions \ref{slnminproblemJmulambda} and \ref{2ndnonnegativesln} below).} 
    
\end{remark}

\vskip 6pt

We study problem \eqref{HS1} through the so-called {\it reduction by inversion} (see \cite{Melo-dosSantos2015}, \cite{Guimaraes-dosSantos2023},  
 and references therein). For the sake of completeness, we present here some ideas from these references.

\vskip 3pt
First observe that if the pair $(u,v)$ is a smooth solution of the differential equations in \eqref{HS1} {with} $u\geq 0$ (or with $v\geq 0$) in $\Omega$, by virtue of the maximum principle, $v\geq 0$ (respectively $u\geq 0$). In view of this comment, we consider the modified elliptic system
\begin{eqnarray}\label{HS1+} 
\left\{
\begin{aligned}
-\Delta u  & = \lambda v_+^{r} +v_+^{p} \qquad &\hbox{in} \ \ \Omega ,\\
-\Delta v   & = \mu |u|^{s-1}u +|u|^{q-1}u \qquad &\hbox{in} \ \ \Omega ,\\
u &=v = 0 \qquad \quad &\hbox{on} \quad \partial \Omega,
\end{aligned}
\right.
\end{eqnarray}
where $\lambda,\mu \geq 0$ and $\zeta_+=\max\{0,\zeta\}$ for $\zeta\in \R$.

\vskip 3pt
Define $f_+:[0,\infty)\times\R\to\R$ by 
\begin{equation*}\label{def:f+}
f_+(\lambda,\zeta):=\lambda \zeta_+^{r} + \zeta_+^p \qquad \hbox{for} \quad \lambda \in [0,\infty) \quad \hbox{and}\quad \zeta \in \R
\end{equation*}
and denote $F_+(\lambda,\zeta):=\int_0^{\zeta}f_+(\lambda,\hat{\zeta})d\hat{\zeta}$ for $\zeta \in \R$.

\vskip 3pt

\vskip 3pt
Let $\mu\geq 0$ arbitrary, but fixed. Consider $\psi(\mu,\cdot):\R\to\R$ the inverse\label{pgref:psi} mapping of the function $\zeta\mapsto \mu |\zeta|^{s-1}\zeta + |\zeta|^{q-1}\zeta$. Applying $\psi(\mu,\cdot)$ on both sides of the second equation in \eqref{HS1+} we obtain the fourth order equation
\begin{equation}\label{4thorderBVP}
\Delta \big(\psi(\mu, \Delta v)\big) = {f_+}(\lambda,v) \quad \hbox{in} \quad \Omega.
\end{equation}

\vskip 4pt

The boundary conditions in \eqref{HS1+} read now as the {\it Navier boundary conditions}
\begin{equation}\label{eqn:navier_bd_cond}
v=\Delta v =0 \quad \hbox{on} \quad \partial \Omega.
\end{equation}


\vskip 3pt
We then study problem \eqref{HS1} via problem \eqref{4thorderBVP}-\eqref{eqn:navier_bd_cond}. With this in mind, we introduce the spaces
$$
\widetilde{W}:=W^{2,\frac{p+1}{p}}(\Omega)\cap W_0^{1,\frac{p+1}{p}}(\Omega) \quad \hbox{with} \quad \|u\|_{\widetilde{W}}:=\|\Delta u\|_{L^{\frac{p+1}{p}}(\Omega)}
$$
and 
$$
W:=W^{2,\frac{q+1}{q}}(\Omega)\cap W_0^{1,\frac{q+1}{q}}(\Omega) \quad \hbox{with} \quad \|v\|_{W}:=\|\Delta v\|_{L^{\frac{q+1}{q}}(\Omega)}.$$

Theorem 9.17 in \cite{GilbargTrudinger} yields the equivalence of the norms $\|\cdot\|_{\widetilde{W}}$ and $\|\cdot\|_{W^{2,\frac{p+1}{p}}(\Omega)}$ in $\widetilde{W}$.  Similarly, the norms $\|\cdot\|_{W}$ and $\|\cdot\|_{W^{2,\frac{q+1}{q}}(\Omega)}$ are equivalent in $W$. Also, it is readily verified that $\widetilde{W},W$ are reflexive and separable Banach spaces. 

\vskip 3pt

By virtue of the classical {\it Sobolev embedding Theorem} for any $m >0$ with
\begin{equation}\label{Sobolevsubcriticalexp_p}
\frac{1}{m+1} + 
\frac{1}{p+1} \geq 1 - \frac{2}{N},
\end{equation}
the constant $S_{p,m}$ defined by
\begin{equation}\label{SobolevEmbedding_p}
S_{p,m}^{-\frac{1}{m+1}}:=\inf \limits_{\substack{\tilde{w}\in \widetilde{W}\\ \tilde{w} \neq 0}} \frac{\|\tilde{w}\|_{\widetilde{W}}}{\|\tilde{w}\|_{L^{m+1}(\Omega)}}
\end{equation}
is finite and positive. Besides, the embedding $\widetilde{W}\hookrightarrow L^{m+1}(\Omega)$ is compact whenever the inequality \eqref{Sobolevsubcriticalexp_p} is strict. Similarly for $W$, for any $m> 0$ with \begin{equation}\label{Sobolevsubcriticalexp_q}
\frac{1}{m+1} + 
\frac{1}{q+1} \geq 1 - \frac{2}{N},
\end{equation}
 the constant $S_{q,m}$ defined by
\begin{equation}\label{SobolevEmbedding_q}
S_{q,m}^{-\frac{1}{m+1}}:=\inf \limits_{\substack{w\in W\\ w \neq 0}} \frac{\|w\|_{W}}{\|w\|_{L^{m+1}(\Omega)}}
\end{equation}
is finite and positive and the embedding $W\hookrightarrow L^{m+1}(\Omega)$ is compact whenever the inequality \eqref{Sobolevsubcriticalexp_q} is strict.

\vskip 4pt



{In what follows, in this section, we assume \eqref{cond:crit_hyperbola}.} The concept of weak solution of \eqref{4thorderBVP}-\eqref{eqn:navier_bd_cond} we consider is the following.

\begin{defi}
A function $w\in W$ is a weak solution of \eqref{4thorderBVP}-\eqref{eqn:navier_bd_cond} if for every $\varphi \in W$, 
\begin{equation}\label{def:weaksln4thordereqn}
\int_{\Omega}\psi(\mu,\Delta w)\Delta \varphi dx = \int_{\Omega}{f_+(\lambda,w)\varphi} dx.
\end{equation}
\end{defi}

{Using\eqref{cond:crit_hyperbola},} it is readily checked that for $w,\varphi \in W$, the integrands in \eqref{def:weaksln4thordereqn} belong to $L^1(\Omega)$ (see Section \ref{sect:Technical_Rslts}, p.\pageref{pg:finite_int}). \\

\vskip 3pt

In order to prove Theorem \ref{theo:bifur_diagram}, we first demonstrate two key results. The first of them states that under hypotheses {\rm (A1)} and {\rm (A2)}, to every weak solution of the BVP \eqref{4thorderBVP}-\eqref{eqn:navier_bd_cond} there corresponds exactly one classical solution of the system \eqref{HS1}. It also states that the reciprocal holds true. Thus, one can deal with problem \eqref{HS1} by studying \eqref{4thorderBVP}-\eqref{eqn:navier_bd_cond}. The precise statement is as follows. \\

\begin{theorem}\label{theo:regularity_4thorder}
Assume {\rm (A2)}.
\begin{itemize}
\item[i.] Let $u,v\in C^2(\Omega)\cap C(\overline{\Omega})$ be such that $(u,v)$ is a classical solution of the system \eqref{HS1}. Then, 
$$
u\in \widetilde{W}, \quad  v\in W,\quad  u=-\psi(\mu,\Delta v) \quad \hbox{in} \quad \Omega
$$ 
and \eqref{4thorderBVP}-\eqref{eqn:navier_bd_cond} are satisfied pointwise.

\item[ii.] Conversely, let $v\in W$ be a weak solution of \eqref{4thorderBVP}-\eqref{eqn:navier_bd_cond} and set $u=-\psi(\mu,\Delta v)$ a.e. in $\Omega$. Then $u\in \widetilde{W}$, $v\in W$ and the pair $(u,v)$ solves the differential equations in \eqref{HS1} in pointwise sense a.e. in $\Omega$.

\vskip 3pt
\item[iii.] Let $v\in W$ be a weak solution of \eqref{4thorderBVP}-\eqref{eqn:navier_bd_cond} and set $u=-\psi(\mu,\Delta v)$ a.e. in $\Omega$. If in addition we assume {\rm (A1)}, $u,v\in C^2(\overline{\Omega})$ and the pair $(u,v)$ is a classical solution of \eqref{HS1}.
\end{itemize}
\end{theorem}

Our approach to \eqref{4thorderBVP}-\eqref{eqn:navier_bd_cond}, in turn, is variational, and in this direction we introduce the energy $J_{\mu, \lambda}:W \to \R$ defined by
\begin{equation}\label{eq:energy_4thBVP}
J_{\mu,\lambda}(w):=\int_{\Omega}\Big(\Psi(\mu,\Delta w) - F_+(\lambda,w)\Big) dx \quad \hbox{for} \quad w\in W, 
\end{equation}
where $\Psi(\mu,\theta):=\int_{0}^{\theta}\psi(\mu,\hat{\theta})d \hat{\theta}$ for $\theta\in \R$. We notice that $J_{\mu,\lambda}\in C^1(W)$ (see Section \ref{sect:Technical_Rslts} below) with 
\begin{equation}\label{eq:energy_4thBVP_der}
DJ_{\mu,\lambda}(w)\varphi = \int_{\Omega}\psi(\mu,\Delta w)\Delta \varphi
dx - \int_{\Omega}{f_+(\lambda, w) \varphi} dx \quad \hbox{for} \quad w,\varphi \in W.
\end{equation}
Thus,  weak solutions of \eqref{4thorderBVP}-\eqref{eqn:navier_bd_cond} in  $W$ are exactly the critical points of $J_{\mu,\lambda}$ in $W$.

\vskip 3pt

The second major result we establish to prove Theorem \ref{theo:bifur_diagram} is motivated by the celebrated result due to Brezis and Nirenberg in \cite{BREZISNIRENBERG1993}, and it states that local minima of the functional $J_{\mu,\lambda}$, in the class of smooth functions, are also local minima in $W$. Consider the space 
$$
C^2_0(\overline{\Omega}):=\{w\in C^2(\overline{\Omega})\,:\, w=0 \quad \hbox{on} \quad \partial \Omega\}
$$
endowed with the usual norm in $C^2(\overline{\Omega})$. \\

\begin{theorem}\label{theo:local_min}
Let $\lambda,\mu\geq 0$ be fixed. Assume {\rm (A1)} and {\rm (A2)}. Let $w_0\in C^2_0(\overline{\Omega})$ be a local minimum of $J_{\mu,\lambda}$ in $C^2_0(\overline{\Omega})$, i.e., there exists $\delta>0$ such that for any $\varphi \in C^2_0(\overline{\Omega})$,
$$
\|w_0-\varphi\|_{C^2(\overline{\Omega})} \leq \delta \quad \Rightarrow\quad J_{\mu,\lambda}(\varphi) \geq J_{\mu,\lambda}(w_0).
$$

Then, $w_0$ is a local minimum of $J_{\mu,\lambda}$ in $W$, i.e., there exists $\eta>0$ such that for any $w \in W$, 
$$
\|w_0-w\|_{W} \leq \eta \quad \Rightarrow\quad J_{\mu,\lambda}(w) \geq J_{\mu,\lambda}(w_0).
$$
\end{theorem}

{The proof of Theorem \ref{theo:local_min} is carried out for the energy $J_{\mu,\lambda}$ (see Section 4 below). We  point out that it can be readily adapted to consider more general scenarios. The key is to improve the regularity of the critical points of the corresponding energy.} \\



\vskip 4pt

{The proofs of Theorems \ref{theo:regularity_4thorder} and \ref{theo:local_min}, in turn, are based  on the study of regularity  of {\it very weak solutions} (see e.g \cite{GilbargTrudinger}, \cite{BREZISCASENAVEMARTEL1996}, \cite{DiFrattaFiorenza2020}) to {\it elliptic systems} of the form
\begin{equation}\label{S1} 
\left\{
\begin{aligned}
-\Delta u  & = f(x,u,v)
 &\hbox{in}& \quad \Omega,\\
-\Delta v  & = g(x,u,v)
 &\hbox{in}& \quad \Omega,\\
u &=v = 0 &\hbox{on}& \quad \partial \Omega,
\end{aligned}
\right.
\end{equation}
where $f,g\in C^{0,\alpha}_{loc}(\Omega \times \R \times \R)$ for some $\alpha\in (0,1 )$. For the sake of clarity we include here some terminology. Following \cite{DiFrattaFiorenza2020}, a {\it very weak solution} of \eqref{S1} is a pair $(u,v)$ of functions $u$ and $v$ such that $u,v\in L^1_{loc}(\Omega)$, $f(x,u,v)$ and  $g(x,u,v)$ also belong to $L^1_{loc}(\Omega)$ and for any $\varphi, \psi\in C^{\infty}_c(\Omega)$,
\begin{equation*}
-\int_{\Omega}u\Delta \varphi dx = \int_{\Omega}f(x,u,v)\varphi dx
\end{equation*}  
and 
\begin{equation*}
-\int_{\Omega}v\Delta \psi dx = \int_{\Omega}g(x,u,v)\psi dx.
\end{equation*}

\vskip 3pt
To study regularity of very weak solutions of \eqref{S1}, we consider the following hypothesis. Let $p,q>0$ be such that
\begin{itemize}
\item[{\rm (B)}]  for some $c>0$, for any $x\in \Omega$ and any $u,v\in \R$,
\begin{equation}\label{ineq:growth_f}
|f(x,u,v)|\leq c\big(1 + |u|^{\frac{q+1}{p+1}} + |v|\big)^p;
\end{equation}

\begin{equation}\label{ineq:growth_g}
|g(x,u,v)|\leq c\big(1 + |u| + |v|^\frac{p+1}{q+1}\big)^q .
\end{equation}
\end{itemize}

\begin{remark} {Assumption {\rm (B)} is satisfied, for instance, in the case that $f=f(x,v)$ and $g=g(x,u)$ are such that 
$$
|f(x,v)| \leq c\big(1 + |v|^p\big), \quad \hbox{and} \quad |g(x,u)|\leq c\big(1 + |u|^q\big).
$$
In the case $p=q$, {\rm (B)} is also satisfied when $f=f(x,u,v)$ and $g=g(x,u,v)$ are such that 
$$
|f(x,u,v)|+|g(x,u,v)|\leq c\big(1+|u|^p + |v|^p\big).
$$}
\end{remark}

\begin{remark}
    {Hypotheses (A1) and (B) are the basic assumptions for the study of sub-critical Hamiltonian systems. We refer the reader to \cite{felmerdefigueiredo1994}, \cite{DeFigueiredo1996} and \cite{HulshofVanderVorst1993} and references there in.}
\end{remark}

A key result, which we use in the {proofs} of Theorems \ref{theo:regularity_4thorder} and \ref{theo:local_min}, is the following statement about the regularity of very weak solutions to \eqref{S1}.

\vskip 3pt

\begin{theorem} \label{theo:Theorem_1}
Assume {\rm (A1)} and {\rm (B)} and let $(u,v)$ be a very weak solution of \eqref{S1} with $u\in L^{q+1}_{loc}(\Omega)$ and $v\in L^{p+1}_{loc}(\Omega)$. Then 
\begin{itemize}
\item[i.] {$u\in W^{2,\frac{p+1}{p}}_{loc}(\Omega)$, $v\in W^{2,\frac{q+1}{q}}_{loc}(\Omega)$} and the differential equations in \eqref{S1} are satisfied pointwise a.e. in $\Omega$. Even more, whenever $u\in L^{q+1}(\Omega)$ and $v\in L^{p+1}(\Omega)$, then $u\in W^{2,\frac{p+1}{p}}(\Omega)$ and $v \in W^{2,\frac{q+1}{q}}(\Omega)$;

\item[ii.] {if $u\in W^{1,\frac{p+1}{p}}_0(\Omega)$ and  $v\in W^{1,\frac{q+1}{q}}_0(\Omega)$, then} for some $\beta \in (0,1)$, 
$u,v\in C^{2,\beta}(\overline{\Omega})$ and $(u,v)$ is a classical solution of \eqref{S1}, i.e., {all equalities in \eqref{S1}} are satisfied pointwise in $
\overline{\Omega}$. 
\end{itemize}

\vskip 3pt
Moreover, there exists a constant $\mathcal{C}>0$ such that 
\begin{equation}\label{est:main_apriori}
\|u\|_{C^{2,\beta}(\overline{\Omega})}+\|v\|_{C^{2,\beta}(\overline{\Omega})}\leq \mathcal{C},
\end{equation}
and $\beta\in (0,1)$ and $\mathcal{C}$  are functions of $p$, $q$, $\|u\|_{W^{2,\frac{p+1}{p}}(\Omega)}$, $\|v\|_{W^{2,\frac{q+1}{q}}(\Omega)}$ and the constant $c>0$ in \eqref{ineq:growth_f} and \eqref{ineq:growth_g} and, {even more,} $\beta$ and $\mathcal{C}$ map bounded sets onto bounded sets.
\end{theorem}

For the sake of completeness, we include (in Section 2 below) a proof of this theorem that refines a regularity result in \cite{HulshofVanderVorst1993}.

\begin{remark}
As we shall see, the proof of Theorem \ref{theo:Theorem_1}  proceeds by a boot-strap argument. This will reveal that part {\it i.} in Theorem \ref{theo:Theorem_1} still holds replacing (A1) by
\begin{equation*}\label{eq:crit-hyperb_1}
\frac{1}{p+1} + \frac{1}{q+1} = 1 - \frac{2}{N}.
\end{equation*}
\end{remark}


\vskip 6pt

The paper is organized as follows. Section \ref{sect:proof_theo_1} presents the detailed proof of Theorem \ref{theo:Theorem_1}. Section \ref{sect:Technical_Rslts} deals with the technical results needed in the variational setting to study \eqref{4thorderBVP}-\eqref{eqn:navier_bd_cond} (some of the details are postponed to the Appendix at the end of this work) and with the proof of Theorem \ref{theo:regularity_4thorder}.  Section \ref{sect:local_minima} presents the detailed proof of Theorem \ref{theo:local_min}. In Section \ref{cond:PS} we verify that the energy $J_{\mu,\lambda}$ satisfies the Palais-Smale condition and Sections \ref{sect:prelim_mult} and \ref{sect:proof_theo_bif} are devoted to the {proofs of Theorems \ref{theo:Esp-caseMainTheo} and \ref{theo:bifur_diagram}.}

\section{Proof of Theorem \ref{theo:Theorem_1}}\label{sect:proof_theo_1}

Assume {\rm (A1)} and {\rm (B)}. In this section we prove Theorem \ref{theo:Theorem_1} using a bootstrap argument. Throughout the proof $\mathcal{K}>0$ stands for a universal constant depending exclusively on $\Omega$, $q$, $p$, $f$ and $g$.

\vskip 3pt
\underline{\it Step 1.} Assume that {$u\in L^{q+1}_{loc}(\Omega)$, $v\in L^{p+1}_{loc}(\Omega)$} and the pair $(u,v )$ is a very weak solution of \eqref{S1}. Then, {$|u|^{\frac{q+1}{p+1}p}, |v|^{p}\in L^{\frac{p+1}{p}}_{loc}(\Omega)$} and from \eqref{ineq:growth_f},
{$|f(x,u,v)|\in L^{\frac{p+1}{p}}_{loc}(\Omega)$}. Similarly, using \eqref{ineq:growth_g}, {$g(x,u,v)\in L^{\frac{q+1}{q}}_{loc}(\Omega)$}.

\vskip 3pt
From Theorem 3 in \cite{DiFrattaFiorenza2020}, $u\in W_{loc}^{2,\frac{p+1}{p}}(\Omega)$ and $u$ solves, pointwise a.e. in $\Omega$, the first equation in \eqref{S1}. Arguing analogously, using the second equation in \eqref{S1}, $v\in W_{loc}^{2,\frac{q+1}{q}}(\Omega)$ and $v$ solves pointwise a.e. in $\Omega$ the second equation in \eqref{S1}. 

\vskip 3pt

{Now assume that $u\in W^{1,\frac{p+1}{p}}_0(\Omega)$ and that $v\in W^{1,\frac{q+1}{q}}_0(\Omega)$. Theorem 9.15 in \cite{GilbargTrudinger} implies that $u\in \widetilde{W}$ and $v \in W$.} 

\vskip 3pt

Summarizing, $u\in \widetilde{W}$, $v\in W$ and $u,v$ solve pointwise a.e. in $\Omega$ the differential equations in \eqref{S1} and this proves part {\it i.}. 

\vskip 3pt
Notice also from Lemma 9.17 in \cite{GilbargTrudinger}, ${\rm (A1)}$, \eqref{ineq:growth_f}, \eqref{ineq:growth_g}, \eqref{SobolevEmbedding_p} and  \eqref{SobolevEmbedding_q} that
\begin{equation}\label{est:apriori_1}
\begin{aligned}
\|u\|^{\frac{p+1}{p}}_{\widetilde{W}} + \|v\|^{\frac{q+1}{q}}_{W}\leq &\mathcal{K}\big(1 + \|u\|^{q+1}_{L^{q+1}(\Omega)} + \|v\|^{p+1}_{L^{p+1}(\Omega)}\big)\\
\leq & \mathcal{K}\big(1 + S_{p,q}\|u\|^{q+1}_{\widetilde{W}} + S_{q,p}\|v\|^{p+1}_{W}\big).
\end{aligned}
\end{equation}

{Observe also that the above  discussion still holds true if instead of {\rm (A1)}, we assume \eqref{cond:crit_hyperbola}.}

\vskip 3pt

\underline{\it Step 2.} Denote
${\rm p}_0:=p$, ${\rm q}_0:=q$ and 
$$
P_0:=\frac{{\rm p}_0+1}{p}, \quad Q_0:=\frac{{\rm q}_0+1}{q}.
$$

From {\it Step 1}, $u\in W^{2,P_0}(\Omega)$ and $v\in W^{2,Q_0}(\Omega)$. 

\vskip 3pt
We prove the conclusions of the theorem by considering several separate cases. Assume first that $Q_0, P_0\geq \frac{N}{2}$.

\vskip 3pt
From Theorem 7.26 (part (i)) in \cite{GilbargTrudinger}, $\widetilde{W}$ and $W$ are simultaneously, continuously embedded in $L^{m+1}(\Omega)$ for any $m\geq 0$. From {\rm (B)}, $f(x,u,v),g(x,u,v)\in L^{m+1}(\Omega)$. Arguing in a similar manner as in {\it Step 1}, $u,v\in W^{2,m+1}(\Omega)\cap W^{1,m+1}_0(\Omega)$.

\vskip 3pt
Besides, Lemma 9.17 in \cite{GilbargTrudinger} and the Sobolev embeddings in  \eqref{SobolevEmbedding_p} and \eqref{SobolevEmbedding_q} yield
 that for $m>0$,
\begin{equation*}\label{est:apriori_2}
\|u\|_{W^{2,m+1}(\Omega)} + \|v\|_{W^{2,m+1}(\Omega)}\leq \mathcal{K}\big(1 + \|u\|^{\sigma}_{\widetilde{W}} + \|v\|^{\rho}_{W}\big),
\end{equation*}
where 
$$
\sigma= \max\Big\{1,\frac{q+1}{p+1}p,q\Big\}\quad \hbox{and} \quad \rho= \max\Big\{1,\frac{p+1}{q+1}q,p\Big\}.
$$

\vskip 3pt
Fix $m>0$ such that $\frac{N}{2}> m> \frac{N}{2}-1$. From Theorem 7.26 ( part (ii)) in \cite{GilbargTrudinger}, $W^{2,m+1}(\Omega)$ is continuously embedded in $C^{0,\tilde{\alpha}}(\overline{\Omega})$ for some $\tilde{\alpha} \in (0,1)$. Thus, $u,v\in C^{0,\tilde{\alpha}}(\overline{\Omega})$ with
\begin{equation*}\label{est:apriori_3}
\|u\|_{C^{0,\tilde{\alpha}}(\Omega)} \leq \mathcal{K}\|u\|_{W^{2,m+1}(\Omega)} \quad \hbox{and} \quad  \|v\|_{C^{0,\tilde{\alpha}}(\Omega)}\leq \mathcal{K}\|v\|_{W^{2,m+1}(\Omega)}.
\end{equation*}

Since $f,g\in C^{0,\alpha}_{loc}(\Omega \times \R \times \R)$, $f(\cdot,u(\cdot),v(\cdot))$ and $g(\cdot,u(\cdot),v(\cdot))$ belong to $C^{0,\beta}(\overline{\Omega})$ for some $\beta\in (0,\tilde{\alpha})$. 

\vskip 3pt
Theorem 6.24 in \cite{GilbargTrudinger} yields that $u,v\in C^{2,\beta}(\overline{\Omega})$. Even more, Theorem B.1 in \cite{StruweBook2008} together with {\rm (B)} imply that 
\begin{equation}\label{est:apriori_4}
\|u\|_{C^{2,\beta}(\overline{\Omega})} + \|v\|_{C^{2,\beta}(\overline{\Omega})}\leq \mathcal{K}\big(1 + \|u\|_{C^{0,\beta}(\overline{\Omega})}^{\sigma}+\|v\|_{C^{0,\beta}(\overline{\Omega})}^{\rho}\big).
\end{equation}

Thus, the conclusion of the theorem is proven in the case $Q_0,P_0\geq \frac{N}{2}$. Even more, tracing down the estimates \eqref{est:apriori_1}-\eqref{est:apriori_4}, we obtain an estimate of the type \eqref{est:main_apriori}.

\vskip 3pt
\underline{\it Step 3.} Assume that $N\geq 3$ and that $P_0,Q_0<\frac{N}{2}$. In this case, we proceed as follows. From \eqref{eq:crit-hyperb}, we observe that
\begin{equation}\label{ineq:subcrit1}
\frac{NP_0}{N-2P_0} >{\rm q}_0+1 \quad \hbox{and} \quad\frac{NQ_0}{N-2Q_0} > {\rm p_0}+1.
\end{equation}

Set 
\begin{equation}\label{eqn:def_p1}
{\rm p}_1:=\min\Big\{\frac{NQ_0}{N-2Q_0},\frac{p+1}{q+1}\frac{NP_0}{N-2P_0}\Big\} - 1 
\end{equation}
and
\begin{equation}\label{eqn:def_q1}
{\rm q}_1:=\min\Big\{\frac{NP_0}{N-2P_0},\frac{q+1}{p+1}\frac{NQ_0}{N-2Q_0}\Big\}-1
\end{equation}
and  $P_1:=\frac{{\rm p}_1+1}{p}$ and $Q_1:=\frac{{\rm q}_1+1}{q}$. From \eqref{ineq:subcrit1}, $
{\rm p}_1>{\rm p}_0$ and ${\rm q}_1>{\rm q}_0$. 

\vskip 3pt
Next, consider ${\rm k}_0, {\rm r}_0>1$ defined by
\begin{equation}\label{eqn:00}
{\rm p}_1+1:={\rm k}_0({\rm p}_0+1)\quad \hbox{and}
\quad {\rm q}_1+1:={\rm r}_0({\rm q}_0+1).
\end{equation}

\underline{\it Step 3.1.} Let now consider the case $P_1\geq \frac{N}{2}$ and $Q_1<\frac{N}{2}$. From the Sobolev embeddings, for any $m\geq 0$, $W^{2,P_1}(\Omega)$ is continuously embedded in $L^{m+1}(\Omega)$ and hence $u\in L^{m+1}(\Omega)$. Also, $W^{2,Q_1}(\Omega)$ is continuously embedded in $L^{\frac{N Q_1}{N-2Q_1}}(\Omega)$, so that $v\in L^{\frac{N Q_1}{N-2Q_1}}(\Omega)$.
 
\vskip 3pt
Set 
$$
{\rm p}_2:=\frac{N{Q}_1}{N-2Q_1} - 1  \quad \hbox{and} \quad {\rm q}_2:= \frac{q+1}{p+1}({\rm p}_2+1)-1
$$
and define
$$
{P}_2:= \frac{{\rm p}_2+1}{p} \quad \hbox{and} \quad Q_2:=\frac{{\rm q}_2+1}{q}. 
$$

Using that $Q_1>Q_0$, ${\rm p}_2>{\rm p}_1$ and ${\rm q}_2>{\rm q}_1$, we write
$$
{\rm p}_2+1={\rm k}_1({\rm p}_1+1) \quad \hbox{and} \quad {\rm q}_2+1={\rm r}_1({\rm q}_1+1),
$$
where ${\rm k}_1,{\rm r}_1>1$.

\vskip 3pt
Since,
$$
|u|^{\frac{q+1}{p+1}p},\,|v|^p\in L^{P_2}(\Omega) \quad \hbox{and} \quad |u|^q, |v|^{\frac{p+1}{q+1}q} \in L^{Q_2}(\Omega),
$$
we may argue as in {\it Step 1} to find that $u\in W^{2,P_2}(\Omega)\cap W^{1,P_2}_0(\Omega)$ and $v\in W^{2,Q_2}(\Omega)\cap W^{1,Q_2}_0(\Omega)$.

\vskip 3pt
Notice that $P_2>P_1\geq \frac{N}{2}$. If $Q_2 \geq \frac{N}{2}$, proceeding as in {\it Step 2}, we obtain $u,v\in C^{2,\beta}(\overline{\Omega})$ for some $\beta\in (0,1)$.

\vskip 3pt
Observe also that, similarly as we did with estimates \eqref{est:apriori_1}-\eqref{est:apriori_4}, the Sobolev embeddings used in this part can be traced down, to yield an estimate of the type \eqref{est:main_apriori}.

\vskip 3pt
Next, if $Q_2<\frac{N}{2}$, we proceed as follows. From the definition of ${\rm p}_1$ and ${\rm q}_1$,
$$
\frac{N Q_2}{N-2Q_2}>{\rm k}_1\frac{N Q_1}{N-2Q_1}.
$$

Set 
$$
{\rm p}_3:=\frac{N{Q}_2}{N-2{Q}_2} - 1  \quad \hbox{and} \quad {\rm q}_3+1 = \frac{{q}+1}{p+1}({\rm p}_3+1)
$$
and define
$$
{P}_3:= \frac{{\rm p}_3+1}{p} \quad \hbox{and} \quad Q_3:=\frac{{\rm q}_3+1}{q}. 
$$

Since ${\rm p}_3+1>{\rm k}_1({\rm p}_2+1)$ and ${\rm q}_3>{\rm q}_2$, we write
$$
{\rm p}_3+1:={\rm k}_2({\rm p}_2+1)\quad \hbox{and}
\quad {\rm q}_3+1:={\rm r}_2({\rm q}_2+1),
$$
where ${\rm k}_2,{\rm r}_2>1$ with ${\rm k}_2>{\rm k}_1$.

\vskip 3pt
Observe also that
$$
|u|^{\frac{q+1}{p+1}p},\,|v|^p\in L^{P_3}(\Omega)\quad \hbox{and} \quad |u|^q, |v|^{\frac{p+1}{q+1}q} \in L^{Q_3}(\Omega),
$$
and hence arguing as in {\it Step 1}, $u\in W^{2,P_3}(\Omega)\cap W^{1,P_3}_0(\Omega)$ and $v\in W^{2,Q_3}(\Omega)\cap W^{1,Q_3}_0(\Omega)$.

\vskip 3pt
Since $P_3>P_2>\frac{N}{2}$, we argue as above. In the case, $Q_3\geq \frac{N}{2}$, proceeding similarly as in {\it Step 2}, $u,v\in C^{2,\beta}(\overline{\Omega})$. 

\vskip 3 pt
In the case $Q_3<\frac{N}{2}$, we may iterate the above procedure. After a finite number of iterations,  we produce {an} increasing finite list of Sobolev exponents $P_k$ and $Q_{k}$ such that $u\in W^{2,P_k}(\Omega)\cap W^{1,P_k}_0(\Omega)$ and $v\in W^{2,Q_k}(\Omega)\cap W^{1,Q_k}_0(\Omega)$, $P_{k+1}>P_k> \frac{N}{2}$ and
$$
\frac{N Q_{k+1}}{N-2Q_{k+1}}>{\rm k}_1\frac{N Q_k}{N-2Q_k}.
$$
 
Thus, for some $k\in \mathbb{N}$ we find that $P_k,Q_k>\frac{N}{2}$ and proceeding as in {\it Step 2}, $u,v\in C^{2,\beta}(\overline{\Omega})$ for some $\beta\in (0,1)$. Also, a tracing down of the Sobolev embeddings used yields an estimate of the type \eqref{est:main_apriori}.

\vskip 3pt
This proves the result in the case $P_1\geq \frac{N}{2}$ and $Q_1<\frac{N}{2}$. A similar and symmetric argument proves the result in the case $P_1< \frac{N}{2}$ and $Q_1\geq \frac{N}{2}$.

\vskip 3pt
\underline{\it Step 3.2.} Assume now that 
$$
\frac{N}{2}> \max\left\{P_1,Q_1\right\}.
$$

A direct calculation {with $k_0,r_0$ defined in \eqref{eqn:00},} shows that 
\begin{equation}\label{ineq:Sobolev_exp_1}
\frac{N P_1}{N-2P_1}>{\rm r}_0\frac{N P_0}{N-2P_0} \quad \hbox{and} \quad  \frac{N Q_1}{N-2Q_1}>{\rm k}_0\frac{N Q_0}{N-2Q_0}.
\end{equation}

Thus,
$$
\frac{NP_1}{N-2P_1} >{\rm q}_1+1 \quad \hbox{and} \quad\frac{NQ_1}{N-2Q_1} > {\rm p_1}+1.
$$

Set 
$$
{\rm p}_2:=\min\Big\{\frac{NQ_1}{N-2Q_1},\frac{p+1}{q+1}\frac{NP_1}{N-2P_1}\Big\} - 1
$$
and
$$
{\rm q}_2:=\min\Big\{\frac{NP_1}{N-2P_1},\frac{q+1}{p+1}\frac{NQ_1}{N-2Q_1}\Big\}-1.
$$

We also set $P_2:=\frac{{\rm p}_2+1}{p}$ and $Q_2:=\frac{{\rm q}_2+1}{q}$. From \eqref{eqn:def_p1}, \eqref{eqn:def_q1} and  \eqref{ineq:Sobolev_exp_1}, $
{\rm p}_2+1 >{\rm p}_1+1$ and ${\rm q}_2+1 >{\rm q}_1+1$.

\vskip 3pt
Thus, we can consider ${\rm k}_1,{r}_1>1$ defined by
$$
{\rm p}_2+1:={\rm k}_1({\rm p}_1+1)\quad \hbox{and}
\quad {\rm q}_2+1:={\rm r}_1({\rm q}_1+1)
$$
with ${\rm k}_1>{\rm k}_0$ and ${\rm r}_1>{\rm r}_0$.

\vskip 3pt
Assuming again that $\frac{N}{2}>\max \{P_2,Q_2\}$ (otherwise we argue as in {\it Step 3.1}) and noticing that  ${\rm p}_2>{\rm k}_0 {\rm p}_1$ and ${\rm q}_2>{\rm r}_0 {\rm q}_1$, 
$$
\frac{N P_2}{N-2P_2}>{\rm k}_0\frac{N P_1}{N-2P_1}\quad \hbox{and} \quad\frac{N Q_2}{N-2Q_2}>{\rm r}_0\frac{N Q_1}{N-2Q_1}.
$$

Next, we iterate the above procedure. This iterative procedure stops after a finite number of iterations yielding Sobolev exponents $P_k$ and $Q_k$ 
so that 
$$
\frac{N P_{k+1}}{N-2P_{k+1}}>{\rm k}_0\frac{N P_k}{N-2P_k}\quad \hbox{and} \quad\frac{N Q_{k+1}}{N-2Q_{k+1}}>{\rm r}_0\frac{N Q_k}{N-2Q_k}.
$$

Thus, 
$$
\frac{N}{2}\leq \min\{P_k,Q_k\}.
$$

We can then proceed as in the {\it Step 2} to find that $u,v\in C^{2,\beta}(\overline{\Omega})$ for some $\beta \in (0,1)$ and $u$ and $v$ are classical solutions to \eqref{S1} with an estimate of the type \eqref{est:main_apriori}. This proves the result in the case $P_1,Q_1<\frac{N}{2}$ and concludes the proof of the theorem. 

\QEDB

\section{Technical Results}\label{sect:Technical_Rslts}
In this section we describe some of the technical results used in our upcoming developments. For the sake of clarity, we postpone the detailed proofs to the appendix.

\vskip 3pt
Assume throughout hypothesis {\rm (A2)}. For $\lambda, \mu \geq 0$ and $\zeta\in \R$, write
$$
f(\lambda,\zeta):= \lambda |\zeta|^{r-1}\zeta + |\zeta|^{p-1}\zeta, \qquad F(\lambda,\zeta):=\int_0^{\zeta}f(\lambda,\zeta)d\zeta
$$
and
$$
g(\mu,\zeta):= \mu|\zeta|^{s-1}\zeta + |\zeta|^{q-1}\zeta, \qquad G(\mu,\zeta):= \int_{0}^{\zeta} g(\mu, \zeta)d\zeta.
$$

Observe that 
\begin{itemize}
\item[i.] $F,G\in  C^1(\R\times \R)\cap C^{\infty}(\R \times \R\setminus\{0\})$;

\vskip 3pt
\item[ii.] for $\lambda,\mu\in [0,\infty)$ fixed, the functions $F(\lambda,\cdot)$ and $G(\mu,\cdot)$ are even and convex;

\vskip 3pt
\item[iii.] for $\lambda,\mu\in [0,\infty)$ fixed, $f(\lambda,\cdot),g(\mu, \cdot):\R\to \R$ are strictly increasing;

\vskip 3pt

\item[iv.] for $\lambda,\mu\in (0,\infty)$ fixed, $f(\lambda,\cdot)$ and $g(\mu,\cdot)$ have {\it concave} geometry near zero, i.e., 
\begin{equation*}\label{eq:conc_conv_geom1}
\lim \limits_{\zeta\to 0} \frac{f(\lambda,\zeta)}{|\zeta|^{r-1}\zeta}=\lambda \quad \hbox{and} \quad \lim \limits_{\zeta\to 0} \frac{g(\mu,\zeta)}{|\zeta|^{s-1}\zeta}=\mu.  
\end{equation*}
\end{itemize}

\begin{remark}We stress that the discussion below is carried out for the function $g$, but an identical analysis can be done for the function $f$. 
\end{remark}

Next, consider the function $\psi:[0,\infty) \times \R \to \R$, which for any fix $\mu \geq 0$
\begin{equation*}
\label{eq:inv_g_mu}
\psi (\mu ,\cdot):\R \to \R, \quad \zeta=\psi(\mu,\theta)
\end{equation*}denotes the inverse function of $g(\mu,\cdot)$. 

\vskip 3pt
{Assume for the moment that $q>1>s$.} Let $\zeta_\mu:= \left[\frac{s(1-s)}{q(q-s)}\right]^{\frac{1}{q-s}}\mu^{\frac{1}{q-s}}$ be the only non-negative number such that
$$
{\partial^2_{\zeta} g(\mu,\zeta_\mu) =\mu s(s-1)|\zeta_\mu|^{s-3}\zeta_\mu + q(q-1)|\zeta_\mu|^{q-3}\zeta_\mu=0}
$$
and
$$
\partial^2_{\zeta} g(\mu, \zeta)(\zeta - \zeta_\mu) >0 \quad \hbox{for} \quad \zeta >0.
$$ 

Notice that when $\mu>0$, $\psi(\mu,\cdot)$ changes {from convex to concave,} exactly once in $(0,\infty)$ at the point 
\begin{equation}\label{theta*}
\theta_\mu:=g(\mu,\zeta_\mu)= C_{q,s}\mu^{\frac{q}{q-s}},
\end{equation}
with $C_{q,s}:=\left[\frac{s(1-s)}{q(q-s)}\right]^{\frac{s}{q-s}} + \left[\frac{s(1-s)}{q(q-s)}\right]^{\frac{q}{q-s}}.$ When $\mu=0$ or $1>q>s>0$, we set $\zeta_\mu=\theta_\mu=0$. Observe that $\psi(0,\theta)=|\theta|^{\frac{1}{q}-1}\theta$ for $\theta\in \R$.

\vskip 3pt
In the rest of this section we present some technical lemmas, whose purpose is to collect relevant properties of $\psi(\mu,\cdot)$. 
\begin{lemma}\label{asymptoticpsizeroinfinity}
For every $\mu \geq 0$:
\begin{itemize}
\item[(i)] {$\psi (\mu, \cdot)\in C^1(\R\setminus\{0\})\cap C(\R)$ with $\psi(\mu, 0)=0$. If $q\geq 1$, then $\psi(\mu,\cdot)\in C^{0,\frac{1}{q}}(\R)$. If either $\mu>0$ or $q\in (0,1)$, then $\psi(\mu,\cdot)\in C^1(\R)$ and $\partial_{\theta} \psi (\mu,0)=0$;}

\item[(ii)] $\psi(\mu,\cdot)$ is strictly increasing and odd;

\item[(iii)] $\psi(\mu, \theta)\theta >0$ for any $\theta \neq 0$;

\item[(iv)] {$\lim \limits_{|\theta|\to \infty} \frac{\psi(\mu,\theta)}{|\theta|^{\frac{1}{q}-1}\theta}=1 $ and for $\mu>0$, $\lim \limits_{|\theta|\to 0} \frac{\psi(\mu,\theta)}{|\theta|^{\frac{1}{s}-1}\theta}=\mu^{-\frac{1}{s}}$.}
\end{itemize}

Even more, $\psi\in C^1((0,\infty)\times \R)$ and for $\mu>0$ and $\theta \neq 0$
\begin{equation*}\label{eqn:der_psi_mu}
\partial_{\mu}\psi(\mu,\theta)=-\frac{|\psi(\mu,\theta)|^{s-1}\psi(\mu,\theta)}{\mu s |\psi(\mu,\theta)|^{s-1}+q|\psi(\mu,\theta)|^{q-1}}
\end{equation*}
and $\partial_{\mu}\psi(\mu,0)=0$.
\end{lemma}

Next, we study some properties of the function $\Psi:[0,\infty)\times \R \to \R$ defined by
$$
\Psi(\mu, \theta) := \int_{0}^{\theta} \psi(\mu, \hat{\theta})d\hat{\theta} \quad\hbox{for}\quad \mu\geq 0, \quad \theta \in \R.
$$

\begin{lemma}\label{asymptoticsPSI}
Let $\mu\geq 0$. The function $\Psi(\mu, \cdot)$ satisfies that:
\begin{itemize}
\item[(i)] {$\Psi(\mu,\cdot)\in C^2(\R\setminus\{0\})\cap C^1(\R)$ with $\Psi(\mu,0)=\partial_{\theta}\Psi(\mu,0)=0$. If $q\geq 1$, then $\Psi(\mu,\cdot)\in C^{1,\frac{1}{q}}(\R)$. If either $\mu>0$ or $q\in (0,1)$, then $\Psi(\mu,\cdot)\in C^2(\R)$ with $\partial^2_{\theta}\Psi(\mu,0)=0$;}

\item[(ii)] $\Psi(\mu,\cdot)$ is strictly convex and even;

\item[(iii)]  $\Psi(\mu, \theta) >0$ for $\theta \neq 0$;

\item[(iv)] for every $\theta \in \R$, $\Psi(\mu,\theta)=  \Psi(\mu,|\theta|).$

\item[(v)] {$\lim \limits_{|\theta| \to \infty} \frac{\Psi(\mu, \theta)}{|\theta|^{\frac{q+1}{q}}} =\frac{q}{q+1}$ and for $\mu>0$, $\lim \limits_{|\theta| \to 0} \frac{\Psi(\mu, \theta)}{|\theta|^{\frac{s+1}{s}}} =\frac{s}{s+1}\mu^{-\frac{1}{s}}$;}

\item[(vi)] The function $\theta \mapsto \frac{\Psi(\mu,\theta)}{\theta}$ is strictly increasing. 
\end{itemize}
\end{lemma}


We finish this discussion by collecting some inequalities, that will be useful in upcoming sections. 

\begin{lemma}\label{inequalitiespsi}
For any $\mu \geq 0$ and $\theta\in \R$,
\begin{equation}\label{ComparisonInequality}
\frac{q}{q+1} \psi(\mu,\theta)\theta \geq  \Psi(\mu,\theta) \geq \frac{s}{s+1} \psi(\mu,\theta)\theta.
\end{equation}

{Moreover, if $\mu>0$}, there exist constants $\hat{C}_{q,s}, \hat{c}_{q,s}\in (0,1)$ such that 
\begin{equation}\label{growthpsi}
|\theta|^{\frac{q+1}{q}}\geq \psi(\mu, \theta)\theta \geq \left\{
\begin{aligned}
\hat{C}_{q,s}|\theta|^{\frac{q+1}{q}}, & \quad  \hbox{for} & |\theta| \geq \theta_\mu\,\\
\mu^{-\frac{1}{s}} \hat{c}_{q,s}|\theta|^{\frac{s+1}{s}}, & \quad  \hbox{for} & |\theta| \leq \theta_\mu,
\end{aligned}
\right.
\end{equation}
where $\theta_\mu$ is defined in \eqref{theta*}.
\end{lemma}

{\begin{remark}\label{rem:I} In the case $\mu=0$, $\psi(0,\theta)=|\theta|^{\frac{1}{q}-1}\theta$ for $\theta \in \R$ and so the inequalities in \eqref{growthpsi} read as $\psi(0,\theta)\theta=|\theta|^{\frac{q+1}{q}}$ for $\theta\in \R$.
\end{remark}}

\vskip 3pt
Next, we state a strong monotonicity property for $\psi(\mu,\cdot)$.

\begin{lemma}\label{technicalPSineq} There exists a constant $C=C_{s,q}>0$ such that given $\mu \geq 0$ and given any $\theta_1, \theta_2 \in \R$ with $\theta_1\neq \theta_2$,
\begin{equation}\label{monotonicitypsi}
\begin{small}
\Bigl(\psi(\mu,\theta_1)- \psi(\mu,\theta_2)\Bigr)(\theta_1 - \theta_2) \geq \frac{C|\theta_1 - \theta_2|^{\frac{s+1}{s}}}{\left(\mu^{\frac{1}{s}} + |\psi(\mu,\theta_1)|^{\frac{q-s}{s}}+ |\psi(\mu,\theta_2)|^{\frac{q-s}{s}}\right)}.
\end{small}
\end{equation}
\end{lemma}

\vskip 3pt
Next, we present the functional analytic framework to study \eqref{4thorderBVP}-\eqref{eqn:navier_bd_cond}. First, we discuss the finiteness of the integrals in \eqref{def:weaksln4thordereqn}. Let $w,\varphi \in W$ be arbitrary. Then,
$$
\begin{aligned}
\int_{\Omega}\Big|\psi(\mu,\Delta w) \Delta \varphi \Big|dx &\leq \Big(\int_{\Omega}\psi(\mu,|\Delta w|)^{q+1}\Big)^{\frac{1}{q+1}}\Big(\int_{\Omega}|\Delta \varphi|^{\frac{q+1}{q}}\Big)^{\frac{q}{q+1}} &\quad &\hbox{from H\"{o}lder Inequality}\\
& \leq \Big(\int_{\Omega}|\Delta w|^{\frac{q+1}{q}}\Big)^{\frac{1}{q+1}}\|\varphi\|_W   &\quad &\hbox{from \eqref{growthpsi} in Lemma \ref{inequalitiespsi}}\\
& \leq \|w\|^{\frac{1}{q}}_W\|\varphi\|_{W}.
\end{aligned}
$$

On the other hand, assuming {\eqref{cond:crit_hyperbola}} and using the H\"older inequality and the Sobolev embeddings (see \eqref{Sobolevsubcriticalexp_q} and \eqref{SobolevEmbedding_q} with $m=p$),
$$
\begin{aligned}
\int_{\Omega}\Big|f_+(\lambda,w) \varphi\Big| dx &\leq C\int_{\Omega}\Big(1+|w|^{p}\Big)|\varphi|dx\\
&\leq 
C\Big(1+\int_{\Omega}|w|^{p+1}dx\Big)^{\frac{p}{p+1}}\Big(\int_{\Omega}|\varphi|^{p+1}dx\Big)^{\frac{1}{p+1}} &\quad &\hbox{from H\"{o}lder Inequality}\\
& \leq C\Big(1+ \|w\|^{p}_W\Big)\|\varphi\|_{W}. &\quad &\hbox{from \eqref{Sobolevsubcriticalexp_q} and \eqref{SobolevEmbedding_q} with $m=p$}
\end{aligned}
$$

This guarantees that the \label{pg:finite_int} integral on the right side of \eqref{def:weaksln4thordereqn} is finite.

\vskip 3pt
\underline{\it Proof of Theorem \ref{theo:regularity_4thorder}.}
Part {\it i.} follows directly from hypothesis (A2) and the fact that $u,v$ are classical solutions of \eqref{HS1}. We now prove part {\it ii.}. Recall that we are assuming \eqref{cond:crit_hyperbola}, which implies the continuity of the embeddings $\widetilde{W}\hookrightarrow L^{q+1}(\Omega)$ and $W\hookrightarrow L^{p+1}(\Omega)$. Let $v\in W$ be a weak solution of \eqref{4thorderBVP}-\eqref{eqn:navier_bd_cond}. Using (A2) and \eqref{cond:crit_hyperbola}, $v\in L^{p+1}(\Omega)$ and it is direct to verify that $h_v:=\lambda |v|^{r-1}v+|v|^{p-1}v \in L^{\frac{p+1}{p}}(\Omega)$. Theorem 9.15 in \cite{GilbargTrudinger} implies the existence of a unique $\widetilde{u}\in \widetilde{W}$ such that $
-\Delta \widetilde{u}=h_v$ a.e. in $\Omega$.

Proceeding analogously, $\widetilde{u}\in L^{q+1}(\Omega)$ and consequently $h_{\widetilde{u}}:=\mu |\widetilde{u}|^{s-1}\widetilde{u}+|\widetilde{u}|^{q-1}\widetilde{u}\in L^{\frac{q+1}{q}}(\Omega)$. Again, Theorem 9.15 in \cite{GilbargTrudinger} implies the existence of a unique $\widetilde{v}\in W$ such that 
$-\Delta \widetilde{v}=h_{\widetilde{u}}$ a.e in $\Omega$. This last equality can be rewritten as $\widetilde{u}=-\psi(\mu,\Delta \widetilde{{v}})$ a.e. in $\Omega$.\\

We claim that $\Delta v=\Delta \widetilde{v}$ a.e. in $\Omega$. To prove the claim observe that for any given $\varphi \in W$
$$
\begin{aligned}
\int_{\Omega} \psi(\mu,\Delta \widetilde{v})\,\Delta \varphi dx &= \int_{\Omega} \big(-\widetilde{u}\big) \,\Delta \varphi dx\\
&= \int_{\Omega} \big(-\Delta \widetilde{u}\big) \,\varphi dx\\
&= \int_{\Omega} h_{v} \,\varphi dx \\
&= \int_{\Omega} \psi(\mu,\Delta {v})\Delta \varphi dx.
\end{aligned}
$$

Taking $\varphi=v-\widetilde{v}\in W$, we find that 
\begin{equation}\label{eqn:tilde=notilde}
\int_{\Omega} \Big(\psi(\mu,\Delta {v})-\psi(\mu,\Delta \widetilde{v})\Big)\,(\Delta v - \Delta \widetilde{v}) dx=0.
\end{equation}

Since $\psi(\mu,\cdot)$ is strictly increasing, the integrand in \eqref{eqn:tilde=notilde} is non-negative. We conclude that $\psi(\mu,\Delta v)=\psi(\mu, \Delta \widetilde{v})$ a.e. in $\Omega$ and the claim follows.\\

The uniqueness stated in Theorem 9.15 in \cite{GilbargTrudinger} implies that $v=\widetilde{v}$ a.e. in $\Omega$ and thus, $v\in W$ and $u:=-\psi(\mu,\Delta v)=-\psi(\mu,\Delta \widetilde{v})\in \widetilde{W}$ solve the differential equations in \eqref{HS1} a.e. in  $\Omega$. This completes the proof of part {\it ii.}.\\

The proof of part {\it iii.} in Theorem \ref{theo:regularity_4thorder} is a consequence of Theorem \ref{theo:Theorem_1} and part {\it ii.}, noticing that when $w\in W$ is a weak solution of \eqref{4thorderBVP}-\eqref{eqn:navier_bd_cond}, $u:=-\psi(\mu,\Delta w)$ and $v:=w$ a.e. in $\Omega$ satisfy that $u\in \widetilde{W}$, $v\in W$ and then the pair $(u,v)$ is a very weak solution of \eqref{HS1}. 

\QEDB

\vskip 3pt
Regarding the energy $J_{\mu,\lambda}$ defined in \eqref{eq:energy_4thBVP}, we first observe that the functional 
$$
W\ni w \mapsto \int_{\Omega}\Psi(\mu,\Delta w)dx \in [0,\infty)
$$
is well defined (see Lemma \ref{inequalitiespsi}) and, even more, it is weakly lower semicontinuous due to the continuity and convexity of the function $\Psi(\mu,\cdot):\R\to [0,\infty)$. 

\vskip 3pt
Besides, the fact that $\Psi(\mu,\cdot)\in C^1(\R)$ (see Lemmas \ref{asymptoticpsizeroinfinity} and \ref{asymptoticsPSI}) and using arguments in the line of those in Remark 3.2.26 in page 127 in \cite{DrabekMilota} (see also Lemma \ref{inequalitiespsi} above), yield that the {functional} 
$$
W\ni v\mapsto \int_{\Omega}\Psi(\mu,\Delta v)dx 
$$
belongs to $C^1(W)$. 

\vskip 3pt

Even more, hypotheses {\rm (A2)} and \eqref{cond:crit_hyperbola} yield that $J_{\mu,\lambda}$ is well defined. Moreover, $J_{\mu,\lambda}\in C^1(W)$ and \eqref{eq:energy_4thBVP_der} holds true.

\section{Proof of Theorem 1.3.}\label{sect:local_minima}

In this section we present the detailed proof of Theorem \ref{theo:local_min}. The proof is based on an adaptation of the proof of Lemma 2.2 in \cite{IturrUbillBrock2008}. Assume {\rm (A1)} and {\rm (A2)}. 

\vskip 3pt
Arguing by contradiction, there are sequences $\{w_n\}_n\subset W$ and $\{\epsilon_n\}_n\subset (0,\infty)$ satisfying that
\begin{equation}\label{hyp:contr_min}
\|w_0-w_n\|_W\leq \epsilon_n, \qquad \epsilon_n \to^+ 0\quad \hbox{and} \quad J_{\mu,\lambda}(w_n)< J_{\mu,\lambda}(w_0). 
\end{equation}

The first observation we make is that the sequence $\{w_n\}_n$ is bounded in $W$. 

\vskip 3pt

Let $\sigma, \rho$ be fixed and such that 
\begin{equation*}\label{eqn:sigma_rho}
0<\sigma<\min\Big\{{1},\frac{1}{q},r\Big\} \quad \hbox{and} \quad p\leq \rho<\frac{N\frac{q+1}{q}}{N-2\frac{q+1}{q}}{-1}.
\end{equation*}

and consider the function
$$
G:W\to \R, \qquad G(w):= \int_{\Omega}\big|w -w_0\big|^{\sigma+1}dx+\int_{\Omega}\big|w -w_0\big|^{\rho+1}dx.
$$

Observe that $G\in C^1(W)$ with 
\begin{multline*}
DG(w)\varphi:=(\sigma+1)\int_{\Omega}|w-w_0|^{\sigma-1}(w-w_0)\varphi dx\\
+ (\rho+1)\int_{\Omega}|w-w_0|^{\rho-1}(w-w_0)\varphi dx
\end{multline*}
for $w\,\varphi\in W$.

\vskip 3pt
For $n\in \mathbb{N}$, set $r_n:=S_{q,\sigma}\epsilon_n^{\sigma+1}+S_{q,\rho}\epsilon_n^{\rho+1}$
and define
$$
M_{n}:=\Big\{w\in W\,:\,G(w)\leq  r_n\Big\}.
$$

From \eqref{Sobolevsubcriticalexp_q} and \eqref{SobolevEmbedding_q} with $m=\sigma$ and $m=\rho$, $w_0,w_n\in M_n$. Consider the minimization problem
\begin{equation}\label{min_prblm}
\min\Big\{J_{\mu,\lambda}(w)\,:\,w\in M_n\Big\}.
\end{equation}

\vskip 3pt
Observe that $J_{\mu,\lambda}$ is coercive in $M_n$. Also, the compactness of the embedding $W\hookrightarrow L^{m+1}(\Omega)$ for $m=\sigma$ and $m=\rho$, yields that $M_n$ is sequentially weakly closed in $W$. 

\vskip 3pt
Since $J_{\mu,\lambda}\in C^1(W)$, we readily check {(from {\rm (A1)})} the sequential weak lower semicontinuity of $J_{\mu,\lambda}$. Thus, the problem \eqref{min_prblm} has a solution.

\medskip
Our second observation is that, with no loss of generality, we may assume that $w_n\in W$ solves \eqref{min_prblm}. 

\vskip 3pt
Setting 
$$
u_n:=-\psi(\mu,
\Delta w_n) \quad \hbox{and} \quad v_n:=w_n \quad \hbox{for}\quad n\in \mathbb{N}\cup\{0\},
$$ 
$\{u_n\}_n\subset L^{q+1}(\Omega)$ and $\{v_n\}_n\subset{W}$.

Observe also that
$$
M_n=\big\{G<r_n\big\}\cup \big\{G=r_n\big\}.
$$

Let $n\in \mathbb{N}$ be fixed, but arbitrary. 

\vskip 4pt

\underline{\it Case 1.} If $w_n\in \big\{w\in W\,:\,G(w)<  r_n\big\}$, then $w_n\in W$ is a critical point of $J_{\mu,\lambda}$. The boundedness of $\{w_n\}_n$ in $W$ and Theorems \ref{theo:Theorem_1} and \ref{theo:regularity_4thorder} yield that $u_n,v_n\in C^{2,\beta}(\overline{\Omega})$ for some $\beta\in (0,1)$ independent of $n$. Also, from Theorem \ref{theo:Theorem_1}, there exists a constant $C>0$, independent of $n$, such that
\begin{equation}\label{est:uniformest}
\|u_n\|_{C^{2,\beta}(\overline{\Omega})}+\|v_n\|_{C^{2,\beta}(\overline{\Omega})}\leq \mathcal{C}.
\end{equation}

\vskip 3pt
\underline{\it Case 2.} Now, assume that $w_n\in\big\{w\in W\,:\,G(w)=  r_n\big\}$. Since $r_n>0$ and $G\in C^1(W)$, the {\it Lagrange Multipliers Theorem} (see Theorem 6.3.2. in \cite{DrabekMilota}) yields that for some  $\gamma_n \in \R$,
$$
DJ_{\mu,\lambda}(w_n)=\gamma_n DG(w_n) \quad \hbox{in} \quad W,
$$
i.e., for any $\varphi\in W$,
\begin{multline}\label{eqn:4thord_Lagr_Mul}
\int_{\Omega}\psi(\mu, \Delta w_n)\Delta \varphi dx -\int_{\Omega}f_+(\lambda,w_n)\varphi dx =\\ =\gamma_n\Big[(\sigma+1)\int_{\Omega}\big|w_n-w_0|^{\sigma-1}(w_n-w_0)\varphi dx \\
+(\rho+1)\int_{\Omega}\big|w_n-w_0|^{\rho-1}(w_n-w_0)\varphi dx \Big].
\end{multline}

\vskip 3pt
The pair $(u_n,v_n)$ is a very weak solution of the system \eqref{S1} with
$$
f(x,u,v)=f_+(\lambda,v)+ \gamma_n\Big[(\sigma+1)\big|v-w_0|^{\sigma-1}(v-w_0)+(\rho+1)\big|v-w_0|^{\rho-1}(v-w_0)\Big] 
$$
and 
$$
g(x,u,v)=\mu|u|^{s-1}u + |u|^{q-1}u.
$$

Observe that 
\begin{multline*}\label{est:f_theo1.3}
|f(x,u,v)|\leq \max\{1,|\lambda|\}\big(|v|^r+|v|^p\big)+\\
+ C_{\sigma,\rho}|\gamma_n|\big(|w_0|^\sigma + |w_0|^\rho+ |v|^\sigma + |v|^\rho\big)
\end{multline*}

Hence, hypotheses {\rm (A1)} and {\rm (B)} are guaranteed with $\rho$ in place of $p$. Thus, arguing as in {\it Case 1}, applying Theorem \ref{theo:Theorem_1}, we obtain \eqref{est:uniformest} that {$w_n\in C^2(\overline{\Omega})$.}

\vskip 3pt

\underline{\it Claim.} There exists $c>0$, independent of $n$, such that $|\gamma_n |\leq c$. 

\vskip 3pt
Assume for the moment the claim has been proven. 
Observe that hypothesis {\rm (B)} would be guaranteed with $\rho$ in place of $p$ in \eqref{ineq:growth_f} and with a constant $C>0$,  independent of $n$, in \eqref{ineq:growth_f} and \eqref{ineq:growth_g}.

\vskip 3pt
Thus, arguing as in {\it Case 1}, applying Theorem \ref{theo:Theorem_1}, we obtain \eqref{est:uniformest} with $\beta\in (0,1)$ and $\mathcal{C}>0$ independent of $n$.

\vskip 3pt
In either case, we may apply the {\it Arzela-Ascoli Theorem}, and pass to subsequence if necessary, so that
$$
u_n \to u_0:=-\psi(\mu,\Delta w_0) \quad \hbox{and} \quad v_n \to v_0:=w_0
$$
uniformly in $C^2(\overline{\Omega})$. Thus, $w_n\to w_0$ in $C^2_0(\overline{\Omega})$ and using \eqref{hyp:contr_min}, for $n\in \mathbb{N}$ sufficiently large, 
\begin{equation*}\label{cond:contrd_minimum}
J_{\mu,\lambda}(w_n) < J_{\mu,\lambda}(w_0) \leq J_{\mu,\lambda}(w_n),
\end{equation*}
which is a contradiction.

\vskip 3pt
The rest of the proof is devoted to establish the {\it Claim}. Since $w_n,w_0\in C^2_0(\overline{\Omega})$ and $w_0$ is a local minimum of $J_{\mu,\lambda}$ in $C^2_0(\overline{\Omega})$, 
\begin{equation}\label{eqn:4thord_locl_mintheo1.3}
\int_{\Omega}\psi(\mu,\Delta w_0)\Delta (w_n-w_0)dx -\int_{\Omega}f_+(\lambda,w_0)(w_n-w_0)dx=0
\end{equation}
for any $n\in \mathbb{N}$.

Testing \eqref{eqn:4thord_Lagr_Mul} against $\varphi=w_n-w_0$, using \eqref{eqn:4thord_locl_mintheo1.3} and subtracting, 
\begin{multline}\label{eqn:finalI_theo1.3}
\int_{\Omega}\Big[\psi(\mu,\Delta w_n)-\psi(\mu,\Delta w_0)\Big]\Delta(w_n-w_0)dx - \int_{\Omega}\Big[f_+(\lambda,w_n)-f_+(\lambda,w_0)\Big](w_n-w_0)dx\\
=\int_{\Omega}\gamma_n\Big[(\sigma+1)|w_n-w_0|^{\sigma+1}+(\rho+1)|w_n-w_0|^{\rho+1}\Big]dx.
\end{multline}

Using \eqref{eqn:finalI_theo1.3}, we estimate 
\begin{equation}\label{eqn:final_III}
\begin{small}
\begin{aligned}
\int_{\Omega}\Big[\psi(\mu,\Delta w_n)-\psi(\mu,\Delta w_0)\Big]&\Delta(w_n-w_0)dx 
+ \int_{\Omega}\Big[f_+(\lambda,w_n)-f_+(\lambda,w_0)\Big](w_n-w_0)dx\\
\geq & \Big| \int_{\Omega}\Big[\psi(\mu,\Delta w_n)-\psi(\mu,\Delta w_0)\Big]\Delta(w_n-w_0)dx \\
&\hspace{1cm}- \int_{\Omega}\Big[f_+(\lambda,w_n)-f_+(\lambda,w_0)\Big](w_n-w_0)dx\Big|\\
=&|\gamma_n|\Big[(\sigma+1)\|w_n-w_0\|_{L^{\sigma+1}(\Omega)}^{\sigma+1}+(\rho+1)\|w_n-w_0\|_{L^{\rho+1}(\Omega)}^{\rho+1}\Big]\\
\geq & (\sigma+1)|\gamma_n|G(w_n)\\
=&(\sigma +1)|\gamma_n|r_n\\
=&  (\sigma+1)|\gamma_n|\big(S_{q,\sigma}\epsilon_n^{\sigma+1}+S_{q,\rho}\epsilon_n^{\rho+1}\big)\\
\geq& (\sigma +1)|\gamma_n|S_{q,\sigma}\epsilon^{\sigma+1}.
\end{aligned}
\end{small}
\end{equation}

Next we estimate the integral terms in the first line of \eqref{eqn:final_III}.

\medskip
If $0<r<p\leq 1$, $f_+(\lambda,\cdot)\in C^{0,r}(\R)$ and then for any $\zeta,\tilde{\zeta}\in \R$ such that $|\zeta|,|\tilde{\zeta}|\leq M$,
$$
\begin{aligned}
\big|f_+(\lambda,\zeta)-f_+(\lambda,\tilde{\zeta})\big|&\leq \big|\lambda (\zeta_+^r-\tilde{\zeta}_+^r)\big|+|\zeta_+^p-\tilde{\zeta}_+^p|\\
& \leq \max\{1,|\lambda|\}\Big[|\zeta - \tilde{\zeta}|^r+ |\zeta-\tilde{\zeta}|^p\Big].
\end{aligned}
$$

Using the Sobolev embedding and \eqref{hyp:contr_min},
\begin{equation}\label{est:finalII.1_theo1.3}
\begin{aligned}
\Big|\int_{\Omega}\big(f_+(\lambda,w_n)-f_+(\lambda,w_0)\big)(w_n-w_0)dx\Big|&\leq  {C}_{\lambda,r,p}\Big[\|w_n-w_0\|^{r+1}_{L^{r+1}(\Omega)}+\\
&\hspace{1.5cm}+\|w_n-w_0\|^{p+1}_{L^{p+1}(\Omega)}\Big]\\
&\leq \tilde{C}_{\lambda,r,p}\Big[\|w_n-w_0\|^{r+1}_{W}+\\
&\hspace{1.5cm}+\|w_n-w_0\|^{p+1}_{W}\Big]\\
&\leq \tilde{C}_{\lambda,r,p}(\epsilon_n^{r+1}+\epsilon_n^{p+1})\\
&\leq \mathcal{K}\epsilon_n^{r+1}.
\end{aligned}
\end{equation}

\vskip 3pt
Next, observe that when {$p> 1 >r$}, for any $\zeta,\tilde{\zeta}$ such that $|\zeta|,|\tilde{\zeta}|\leq M$,
$$
\begin{aligned}
\big|f_+(\lambda,\zeta)-f_+(\lambda,\tilde{\zeta})\big|&=\big|\lambda (\zeta_+^r-\tilde{\zeta}_+^r)\big|+|\zeta_+^p-\tilde{\zeta}_+^p|\\
& \leq \max\{1,|\lambda|\}\Big[|\zeta - \tilde{\zeta}|^r+ |\zeta_+^p-\tilde{\zeta}_+^p|\Big]\\
&\leq C_{\lambda,r,p} \Big[|\zeta - \tilde{\zeta}|^r+{(|\zeta|^{p-1}+|\tilde{\zeta}|^{p-1})}|\zeta-\tilde{\zeta}|\Big]
\end{aligned}
$$

Using the H\"older's inequality, the Sobolev embedding and \eqref{hyp:contr_min},
\begin{equation}\label{est:finalII.2_theo1.3}
\begin{aligned}
\Big|\int_{\Omega}\big(f_+(\lambda,w_n)-f_+(\lambda,w_0)\big)(w_n-w_0)dx\Big|&\leq  \tilde{C}_{\lambda,r,p}\Big[\|w_n-w_0\|^{r+1}_{L^{r+1}(\Omega)}+\\
&\hspace{.5cm}+\big(\|w_n\|^{p-1}_{L^{p+1}(\Omega)}+\|w_0\|^{p-1}_{L^{p+1}(\Omega)}\big)\|w_n-w_0\|^2_{L^{p+1}(\Omega)}\Big]\\
&\leq \mathcal{K}\|w_n-w_0\|^{r+1}_{W}\\
&\leq \mathcal{K}\epsilon_n^{r+1}.
\end{aligned}
\end{equation}

On the other hand, if $q\geq 1$, $\psi(\mu,\cdot)\in C^{0,\frac{1}{q}}(\R)$. Using that $\psi$ is strictly increasing, there exists $C=C(\mu)>0$ such that 
\begin{equation}\label{est:finalI_theo1.3}
\begin{aligned}
C \epsilon_n^{\frac{q+1}{q}}&\geq C\|w_n-w_0\|_W^{\frac{q+1}{q}}\\
&=\int_{\Omega}C|\Delta w_n-\Delta w_0|^{\frac{q+1}{q}}dx \\
&\geq \int_{\Omega}|\psi(\mu,\Delta w_n)-\psi(\mu,\Delta w_0)||\Delta w_n-\Delta w_0|dx\\
& = \int_{\Omega}\Big[\psi(\mu,\Delta w_n)-\psi(\mu,\Delta w_0)\Big]\Delta\big( w_n-w_0\big)dx.
\end{aligned}
\end{equation}

Thus, from \eqref{eqn:final_III},
\eqref{est:finalII.1_theo1.3},\eqref{est:finalII.2_theo1.3} and \eqref{est:finalI_theo1.3},
$$
\begin{aligned}
C\epsilon_n^{\frac{q+1}{q}} + \mathcal{K}\epsilon_n^{r+1}&\geq   (\sigma+1)|\gamma_n|S_{q,\sigma}\epsilon_n^{\sigma+1}
\end{aligned}
$$

Therefore,
$$
|\gamma_n|\leq \hat{C}\big(\epsilon_n^{\frac{1}{q}-\sigma}+ \epsilon_n^{{r}-\sigma}\big) \to 0^+ \quad \hbox{as}\quad n\to \infty.
$$

In the case, $q\in (0,1)$, we use the mean value theorem on $\psi(\mu,\cdot)$, the fact that $|\partial_{\theta}\psi(\mu,\theta)|\leq \frac{|\psi(\mu,\theta)|^{1-s}}{\mu(1-s)+(1-q)|\psi(\mu,\theta)|^{q-s}}\leq C|\theta|^{\frac{1-q}{q}}$ and H\"older's inequality to find that 
\begin{equation}\label{est:finalIV_theo1.3}
C\epsilon_n^{2}\geq  C\|w_n-w_0\|_W^{2}\geq \int_{\Omega}|\psi(\mu,\Delta w_n)-\psi(\mu,\Delta w_0)||\Delta w_n-\Delta w_0|dx
\end{equation}
so that \eqref{eqn:final_III},
\eqref{est:finalII.1_theo1.3}, \eqref{est:finalII.2_theo1.3} and \eqref{est:finalIV_theo1.3} yield
$$
|\gamma_n|\leq \hat{C}\big(\epsilon_n^{1-\sigma}+ \epsilon_n^{{r}-\sigma}\big) \to 0^+ \quad \hbox{as}\quad n\to \infty.
$$

In either case, $|\gamma_n|\to 0$ as $n\to \infty$. This proves the claim and completes the proof of the theorem. 
%
%
%
%
%
%
\QEDB

\section{Palais-Smale condition}\label{cond:PS}

Assume hypotheses {\rm (A1)}, {\rm (A2)} and {\rm (A3)} and fix $\lambda,\mu \geq 0$. In this section we study compactness properties of $(PS)-$sequences for the energy $J_{\mu,\lambda}:W \to \R$ defined in \eqref{eq:energy_4thBVP}. Recall also that $J_{\mu,\lambda}\in C^1(W)$ and its derivative $DJ_{\mu,\lambda}$ is given in expression \eqref{eq:energy_4thBVP_der}.

\vskip 3pt
Let $\{w_n\}_{n\in \mathbb{N}}\subset W$ be a Palais-Smale {sequence} ({$(PS)-$sequence} for short) for $J_{\mu,\lambda}$, i.e., $\{J_{\mu,\lambda}(w_n)\}_{n\in\mathbb{N}}$ is bounded and 
$$
D J_{\mu,\lambda}(w_n) \to 0 \quad \hbox{in} \quad W^*, \quad \hbox{as} \quad n\to \infty.  
$$

We prove that $\{w_n\}_{n\in \mathbb{N}}$ has a subsequence that converges strongly in $W$. The proof proceeds by a series of Lemmas.

\vskip 3pt
The first lemma concerns with the boundedness of the $(PS)-$sequence $\{w_n\}_n$ in $W$.

\begin{lemma}\label{PSbounded}
The $(PS)-$ sequence $\{w_n\}_n$ of $J_{\mu,\lambda}$ is bounded in $W$.
\end{lemma}

\begin{proof}
First observe that
\begin{multline}\label{PS1}
J_{\mu,\lambda}(w_n) - \frac{q}{q+1} DJ_{\mu,\lambda}(w_n)w_n=\\
=\underbrace{\int_{\Omega}\left(\Psi(\mu,\Delta w_n)- \frac{q}{q+1} \psi(\mu, \Delta w_n)\Delta w_n \right)dx}_{I_n}\\ 
\underbrace{-\lambda\left(\frac{1}{r+1}-\frac{q}{q+1}\right)\int_{\Omega}(w_n)_+^{r+1}dx - \left(\frac{1}{p+1}-\frac{q}{q+1}\right)\int_{\Omega}(w_n)_+^{p+1}dx}_{II_n}.
\end{multline}

Let $\delta\in (0,1)$ be small, but fixed. Since $\{w_n\}_{n\in\mathbb{N}}$ is a $(PS)-$sequence, there exists $N\in \mathbb{N}$ such that for any $n\geq N$,
\begin{equation}\label{PSderiva}
\begin{aligned}
\left|\int_{\Omega} \psi(\mu, \Delta w_n)\Delta w_n - \lambda\int_{\Omega}(w_n)_+^{r+1}dx - \int_{\Omega}(w_n)_+^{p+1}dx\right|=& \left|DJ_{\lambda,\mu}(w_n)w_n\right| \\
\leq & \delta \|w_n\|_W.
\end{aligned}
\end{equation}

Consequently, for $n\geq N$,
\begin{equation}\label{PS1.2}
\big|I_n+II_n\big| \leq C +  \frac{q\delta}{q+1} \|w_n\|_{W}.
\end{equation}

\medskip

Next, we estimate the integral $I_n$ in \eqref{PS1}. With $\delta>0$ as above and using part $(iv)$ in Lemma \ref{asymptoticpsizeroinfinity} and part $(v)$ in Lemma \ref{asymptoticsPSI}, there exists ${D}>\theta_{\mu}>0$ such that for every $\theta$ with $|\theta|\geq D$,
\begin{equation}\label{PS1.1.1}
\begin{aligned}
\left(1-\delta\right)|\theta|^{\frac{q+1}{q}}\leq&   \psi(\mu,\theta)\theta\leq \left(1+\delta\right)|\theta|^{\frac{q+1}{q}}\\
\frac{q}{q+1}\left(1-\delta\right)|\theta|^{\frac{q+1}{q}}\leq &  \Psi(\mu,\theta)\leq \frac{q}{q+1}\left(1+\delta\right)|\theta|^{\frac{q+1}{q}}.
\end{aligned}
\end{equation}

Using \eqref{PS1.1.1}, we estimate
\begin{equation}\label{PS1.1}
\begin{aligned}
I_n=& 
\int_{\{|\Delta w_n|\geq D\}}\left(\Psi(\mu,\Delta w_n)- \frac{q}{q+1} \psi(\mu, \Delta w_n)\Delta w_n \right)dx \\
&+ \int_{\{|\Delta w_n| <D\}}\left(\Psi(\mu,\Delta w_n)- \frac{q}{q+1} \psi(\mu, \Delta w_n)\Delta w_n \right)dx\\
\geq &\frac{-2q\delta}{q+1}\int_{\{|\Delta w_n|\geq {D}\}}|\Delta w_n|^{\frac{q+1}{q}}dx \\
&+ \int_{\{|\Delta w_n| <{D}\}}\left(\Psi(\mu,\Delta w_n)- \frac{q}{q+1} \psi(\mu, \Delta w_n)\Delta w_n \right)dx\\
\geq &\frac{-2q\delta}{q+1}\int_{\Omega}|\Delta w_n|^{\frac{q+1}{q}}dx - \mathcal{K}
_{\mu,q,s,\delta, \Omega}
\end{aligned}
\end{equation}
for some constant $
\mathcal{K}_{\mu,q,s,\delta,\Omega}>0$. Proceeding similarly as above, and taking $\mathcal{K}
_{\mu,q,s,\delta, \Omega}>0$ larger if necessary,
\begin{equation}\label{PS1.1.1}
I_n\leq \frac{2q\delta}{q+1}\int_{\Omega}|\Delta w_n|^{\frac{q+1}{q}}dx +\mathcal{K}
_{\mu,q,s,\delta, \Omega}.
\end{equation}

\medskip
As for the integral $II_n$ we find that 
\begin{equation}\label{PS1.3}
II_n+\lambda \frac{1-rq}{(r+1)(q+1)}\int_{\Omega}(w_n)_+^{r+1}dx  =  \frac{pq-1}{(p+1)(q+1)}\int_{\Omega}(w_n)_+^{p+1}dx.
\end{equation}

Assume for the moment $qp\neq 1$. We estimate the right-hand side in \eqref{PS1.3}, making use of \eqref{PS1.2}, \eqref{PS1.1} and \eqref{PS1.1.1} to obtain that for any $n\geq N$,
\begin{equation}\label{PS2}
\begin{aligned}
\int_{\Omega}(w_n)_+^{p+1}dx 
\leq &  C + \delta \frac{q}{q+1}\|w_n\|_W + \lambda\frac{(p+1)|1-rq|}{(r+1)(qp-1)}\|w_n\|^{r+1}_{L^{r+1}(\Omega)} \\
& \hspace*{4cm} + \frac{2\delta q(p+1)}{|qp-1|}\|w_n\|^{\frac{q+1}{q}}_{W},
\end{aligned}
\end{equation}
where the constant $C=C(\mu,q,s,\delta,\Omega)>0$ does not depend on $n$.

\medskip
Next, from \eqref{PSderiva}, \eqref{PS2} and the Sobolev embeddings in \eqref{Sobolevsubcriticalexp_q} and \eqref{SobolevEmbedding_q}, we can find a constant $C_1>0$ depending only on $q,p,r$ and $\Omega$ such that for any $n\geq N$,
\begin{equation}\label{EstimateLapPS2}
\int_{\Omega}\psi(\mu, \Delta w_n)\Delta w_n dx \leq C + 2\delta\|w_n\|_W + \lambda C_1\|w_n\|_W^{r+1} + \delta C_1\|w_n\|_W^{\frac{q+1}{q}}. 
\end{equation}

Assume for the moment $\mu>0$. From \eqref{theta*} and \eqref{growthpsi},
\begin{equation}\label{BoundedPSseq2}
\begin{aligned}
\int_{\Omega}\psi(\mu,\Delta w_n)\Delta w_n dx = & \int_{\{|\Delta w_n|\geq \theta_{\mu}\}} \psi(\mu,\Delta w_n)\Delta w_n dx + \int_{\{|\Delta w_n|< \theta_{\mu}\}} \psi(\mu,\Delta w_n)\Delta w_n dx 
\\
\geq & \hat{C}_{q,s} \int_{\{|\Delta w_n|\geq \theta_{\mu}\}} |\Delta w_n|^{\frac{q+1}{q}}dx +\\
& \hspace{75pt} + \underbrace{\mu^{-\frac{1}{s}}c_{q,s}\int_{\{|\Delta w_n|< \theta_{\mu}\}}|\Delta w_n|^{\frac{s+1}{s}}dx}_{\geq 0}\\ 
\geq & \hat{C}_{q,s} \int_{\{|\Delta w_n| \geq  \theta_{\mu}\}} |\Delta w_n|^{\frac{q+1}{q}}dx + \hat{C}_{q,s}\int_{\{|\Delta w_n|< \theta_{\mu}\}} |\Delta w_n|^{\frac{q+1}{q}}dx\\
& \hskip 137pt -\hat{C}_{q,s}\int_{\{|\Delta w_n|< \theta_{\mu}\}} |\Delta w_n|^{\frac{q+1}{q}}dx\\
\geq & \hat{C}_{q,s} \int_{\Omega} |\Delta w_n|^{\frac{q+1}{q}}dx - \hat{C}_{q,s}\int_{\{|\Delta w_n|< \theta_{\mu}\}} |\Delta w_n|^{\frac{q+1}{q}}dx\\
\geq & \hat{C}_{q,s} \|w_n\|^{\frac{q+1}{q}}_W - \hat{C}_{q,s}{C}_{q,s}\mu^{\frac{q+1}{q-s}}{\rm meas}(\Omega).
\end{aligned}
\end{equation}

In the case $\mu=0$, Remark \ref{rem:I} yields the same estimate as in \eqref{BoundedPSseq2} with $\hat{C}_{q,s}=0$.

Since, $\delta\in (0,1)$  is small, we may assume that $\delta<\frac{\hat{C}_{q,s}}{4{C_1}}$, which together with \eqref{EstimateLapPS2} and \eqref{BoundedPSseq2}, yields
\begin{equation}\label{ineq:case_qpneq1}
\|w_n\|_W^{\frac{q+1}{q}}\leq C + c\left(\|w_n\|_W + \|w_n\|_W^{r+1}\right)
\end{equation}
for any $n\geq N$, where the constants $C,c>0$, depend only on $\lambda, \mu, p, q, r,s$ and $\Omega$. 

\vskip 3pt
In the case $qp=1$, taking into account \eqref{PS1.3}, one can get again the estimate \eqref{ineq:case_qpneq1}.

\vskip 3pt
Finally, hypothesis ${\rm (A3)}$ yields that $\{w_n\}_{n\in\mathbb{N}}$ is bounded in $W$. This completes the proof of the lemma.
\end{proof}
{\begin{remark}
From the previous proof we observe that if a bounded sequence $\{\lambda_n\}_{n\in \mathbb{N}}\subset [0,\infty)$ and a sequence $\{w_n\}_{n\in\mathbb{N}}\subset W$ are such that $\{J_{\mu,\lambda_n}(w_n)\}_n$ is uniformly bounded and $DJ_{\mu,\lambda_n}(w_n)w_n/\|w_n\|_W \to 0$, then $\{w_n\}_n$ is uniformly bounded in $W$.
\end{remark}}

At this point, using the reflexivity of $W$ and passing to a subsequence denoted the same, assume that $w_n\rightharpoonup w$ weakly in $W$. 

\begin{lemma}\label{lemma:CC_3}
There exists a constant $C_{s,q,\mu}>0$ such that for any $n\in \mathbb{N}$, 
\begin{equation}\label{StrongconvergW2}
\int_{\Omega}\left(\mu^{\frac{1}{s}} + |\psi(\mu, \Delta w_n)|^{\frac{q-s}{s}}+ |\psi(\mu, \Delta w)|^{\frac{q-s}{s}}\right)^{\frac{s(q+1)}{q-s}}dx \leq C_{s,q,\mu}.
\end{equation}
\end{lemma}

\begin{proof}
From \eqref{growthpsi}  and the fact that  {$\frac{s(q+1)}{q-s}>0$},
\begin{multline*}
\int_{\Omega}\left(\mu^{\frac{1}{s}} + |\psi(\mu, \Delta w_n)|^{\frac{q-s}{s}}+ |\psi(\mu, \Delta w)|^{\frac{q-s}{s}}\right)^{\frac{s(q+1)}{q-s}}dx \\
\leq C_{q,s}\int_{\Omega}\left(\mu^{\frac{q+1}{q-s}} + |\psi(\mu, \Delta w_n)|^{q+1}+ |\psi(\mu, \Delta w)|^{q+1}\right)dx\\
\leq C_{q,s}\int_{\Omega}\left(\mu^{\frac{q+1}{q-s}} + |\Delta w_n|^{\frac{q+1}{q}}+ |\Delta w|^{\frac{q+1}{q}}\right)dx.
\end{multline*}

Since $\{w_n\}_{n\in\mathbb{N}}$ is bounded in $W$, inequality \eqref{StrongconvergW2} follows.
\end{proof}

We finish this section with the following lemma.

\begin{lemma}
$\Delta w_n \to \Delta w$ strongly in $L^{\frac{q+1}{q}}(\Omega)$. 
\end{lemma}

\begin{proof} 
Observe that $w_n - w\in W$ and
\begin{equation}\label{Monotpsimu}
\begin{aligned}
0\leq \int_{\Omega}&\left(\psi(\mu,\Delta w_n)- \psi(\mu, \Delta w)\right) \Delta (w_n -w)dx \\ 
=&\underbrace{\int_{\Omega}\psi(\mu,\Delta w_n)\Delta (w_n -w)dx}_{I_{n}} -  \underbrace{\int_{\Omega} \psi(\mu, \Delta w)\Delta (w_n -w)dx}_{II_{n}}.
\end{aligned}
\end{equation}

Since $\{w_n\}_n$ is bounded in $W$, we find that for $n$ large enough,
\begin{equation}\label{FinitesetJ4.2.0}
I_{n}=\int_{\Omega} f_+(\lambda, w_n)(w_n -w)dx + o(1).
\end{equation}

Since $w_n \to w$ strongly in $L^{p+1}(\Omega)$, from Vainberg's Lemma or the continuity of Charatheodory operators, 
\begin{equation}\label{FinitesetJ4.2}
\begin{aligned}
\lim \limits_{n \to \infty}\int_{\Omega}f_+(\lambda,w_n)(w_n-w)dx 
=&0.
\end{aligned}
\end{equation}

As for the integral term $II_{n}$ in \eqref{Monotpsimu}, the linear functional 
$$
W\ni \varphi \mapsto \int_{\Omega}\psi(\mu,\Delta w)\Delta \varphi dx
$$
is continuous, so that from the weak convergence of $\{w_n\}_{n\in \mathbb{N}}$ in $W$,
\begin{equation}\label{FinitesetJ7.2}
\lim \limits_{n \to \infty}II_{n}=\lim \limits_{n\to \infty}\int_{\Omega}\psi(\mu,\Delta w)\Delta (w_n -w)dx
\to 0.
\end{equation}

It follows from \eqref{Monotpsimu}, \eqref{FinitesetJ4.2.0} \eqref{FinitesetJ4.2} and \eqref{FinitesetJ7.2} that 
\begin{equation}\label{StrongconvergW3}
\int_{\Omega}\left(\psi(\mu,\Delta w_n)- \psi(\mu, \Delta w)\right)\Delta (w_n -w)dx \to 0
\end{equation}
as $n \to \infty$.

\medskip
We finish the proof as follows. For $n\in \mathbb{N}$, set $\Omega_n:=\{\Delta w_n \neq \Delta w\}$. Using H\"{o}lder's inequality with the conjugate exponents {$\frac{q(s+1)}{s(q+1)}$ and $\frac{q(s+1)}{q-s}$, both being larger than one,} we estimate
$$
\int_{\Omega}|\Delta(w_n -w)|^{\frac{q+1}{q}}dx=\int_{\Omega_n}|\Delta(w_n -w)|^{\frac{q+1}{q}}dx=
$$
$$
\begin{aligned}
= \int_{\Omega_n}\frac{\left|\Delta(w_n -w)\right|^{\frac{q+1}{q}}\left(\mu^{\frac{1}{s}} + |\psi(\mu, \Delta w_n)|^{\frac{q-s}{s}}+ |\psi(\mu, \Delta w)|^{\frac{q-s}{s}}\right)^{\frac{s(q+1)}{q(s+1)}}}{\left(\mu^{\frac{1}{s}} + |\psi(\mu, \Delta w_n)|^{\frac{q-s}{s}}+ |\psi(\mu, \Delta w)|^{\frac{q-s}{s}}\right)^{\frac{s(q+1)}{q(s+1)}}}dx 
\end{aligned}
$$
\begin{multline}\label{StrongconvergW1}
\leq
 \left(\int_{\Omega_n}\frac{\left|\Delta(w_n -w)\right|^{\frac{s+1}{s}}}{\left(\mu^{\frac{1}{s}} + |\psi(\mu, \Delta w_n)|^{\frac{q-s}{s}}+ |\psi(\mu, \Delta w)|^{\frac{q-s}{s}}\right)}dx \right)^{\frac{s(q+1)}{q(s+1)}}\times\\
 \left(\int_{\Omega}\left(\mu^{\frac{1}{s}} + |\psi(\mu, \Delta w_n)|^{\frac{q-s}{s}}+ |\psi(\mu, \Delta w)|^{\frac{q-s}{s}}\right)^{\frac{s(q+1)}{q-s}}dx\right)^{\frac{q-s}{q(s+1)}}.
\end{multline}

On the other hand, from Lemma \ref{technicalPSineq}, \begin{multline}\label{PS3}
\left(\int_{\Omega_n}\frac{\left|\Delta(w_n -w)\right|^{\frac{s+1}{s}}}{\left(\mu^{\frac{1}{s}} + |\psi(\mu, \Delta w_n)|^{\frac{q-s}{s}}+ |\psi(\mu, \Delta w)|^{\frac{q-s}{s}}\right)}dx \right)^{\frac{s(q+1)}{q(s+1)}}\leq\\
\leq C_{s,q} \left(\int_{\Omega}\Bigl(\psi(\mu,\Delta w_n)- \psi(\mu, \Delta w)\Bigr)\Delta (w_n -w)dx \right)^{\frac{s(q+1)}{q(s+1)}}.
\end{multline}

Putting together \eqref{StrongconvergW2}, \eqref{StrongconvergW1},\eqref{PS3} we find that
\begin{multline*}
\|\Delta w_n - \Delta w\|_{L^{\frac{q+1}{q}}(\Omega)}
\leq \tilde{C}_{s,q,\mu} \int_{\Omega}\left(\psi(\mu,\Delta w_n)- \psi(\mu, \Delta w)\right)\Delta (w_n -w)dx.
\end{multline*}

The conclusion follows taking limit as $n\to \infty$ in the previous estimate and using \eqref{StrongconvergW3}.
\end{proof}

\section{Preliminary multiplicity results}\label{sect:prelim_mult}

Throughout this section we assume that $qp>1$ and also hypotheses {\rm (A1)}, {\rm (A2)} and {\rm (A3)}. Observe that $r<\frac{1}{q}<p$. We provide some of the preliminary steps in the proof of Theorem \ref{theo:bifur_diagram}. Similar ideas to the ones hereby presented can also be found in \cite{AgudeloKudlacHolubova2024}. First, we present a series of lemmas describing the geometry of the energy $J_{\mu,\lambda}$.

\begin{lemma}\label{lowerbandenergyJ}
There exist constants $c_i>0$, $i=1,2,3$, depending only on $q,p,r,s$ and $\Omega$ such that for any $\mu, \lambda \geq 0$ and for any $w\in W$, $w \neq 0$, 
\begin{equation}\label{lowerboundenergy}
J_{\mu,\lambda}(w) \geq c_1\frac{\|w\|^{\frac{s+1}{s}}_W}{\mu^{\frac{1}{s}}+ \|w\|^{\frac{q-s}{qs}}_W}
- c_2\lambda\|w\|_{W}^{r+1} - c_3\|w\|_W^{p+1}.
\end{equation}
\end{lemma}

\begin{proof}
Let $\mu,\lambda \geq 0$ be arbitrary, but fixed and take $w\in W$, $w\neq 0$. First we write, \begin{equation}\label{lowerboundenergy1}
\begin{aligned}
J_{\mu,\lambda}(w) {=}& \underbrace{\int_{\Omega} \Psi(\mu,\Delta w)dx}_{I} - \underbrace{\int_{\Omega}F_+(\lambda, w)dx}_{II}. 
\end{aligned}
\end{equation}

Using \eqref{ComparisonInequality} from Lemma \ref{inequalitiespsi} and \eqref{monotonicitypsi} from  Lemma \ref{technicalPSineq} and proceeding as we did to obtain \eqref{StrongconvergW1} and \eqref{PS3} (taking $w_n=0$ and following the same lines of the proof of Lemma \ref{lemma:CC_3}), we estimate from below the integral $I$ to find that for some constant $c_1>0$
$$
I=\int_{\Omega}\Psi(\mu,\Delta w)dx \geq c_1\frac{\|w\|^{\frac{s+1}{s}}_W}{\mu^{\frac{1}{s}}+ \|w\|^{\frac{q-s}{qs}}_W}.
$$

Next, we estimate the integral $II$. From {\rm (B)} and the Sobolev em\-beddings in  \eqref{SobolevEmbedding_q}, 
\begin{equation}\label{lowerbound3}
\begin{aligned}
II  =&\frac{\lambda}{r+1}\|{w_+}\|_{L^{r+1}(\Omega)}^{r+1} + \frac{1}{p+1}\|{w_+}\|_{L^{p+1}(\Omega)}^{p+1}\\
\leq & \frac{\lambda S_{q,r}}{r+1}\|w\|_{W}^{r+1} + \frac{S_{q,p}}{q+1}\|w\|_{W}^{p+1}.
\end{aligned}
\end{equation}

From \eqref{lowerboundenergy1} and \eqref{lowerbound3}, we find constants $c_1,c_2,c_3>0$ independent of $\mu$, $\lambda$ and $w\in W$ such that \eqref{lowerboundenergy} holds. This completes the proof.
\end{proof}

In the next lemma, we make use of the notations from Lemma \ref{lowerbandenergyJ}.

\begin{lemma}\label{positiveenergyradius}
\begin{itemize}
\item[(i)] For any $\mu,\lambda \geq 0$, any $R>0$ and for every $w\in W$ with $\|w\|_{W}\leq R$,
$$
J_{\mu,\lambda}(w) \geq - c_2\lambda R^{r+1} - c_3R^{p+1}.
$$

\item[(ii)] There exist $R_0,r_0>0$ with $r_0<R_0$ and there exist $c_0,\mu_0,\lambda_0>0$ such that given any $\mu\in [0,\mu_0]$, $\lambda\in [0,\lambda_0]$ and any $w\in W$ with $r_0 \leq \|w\|_W \leq R_0$,
$$
J_{\mu,\lambda}(w)>c_0.
$$
\end{itemize}
\end{lemma}

\begin{proof}
Part {\it (i)} follows directly from Lemma \ref{lowerbandenergyJ}. Next, we {prove} {\it (ii)}.

Observe from \eqref{lowerboundenergy} that {for any given $w\in W\setminus\{0\}$,}
\begin{equation}\label{ineq:lowerboundenerg}\begin{aligned}
J_{\mu,\lambda}(w) 
\geq & c_1\frac{\|w\|^{\frac{s+1}{s}}_W}{\mu^{\frac{1}{s}}+ \|w\|^{\frac{q-s}{qs}}_W}
- c_2\lambda\|w\|_{W}^{r+1} - c_3\|w\|_W^{p+1}\\
= & c_1\left(1 -\frac{\mu^{\frac{1}{s}}}{\mu^{\frac{1}{s}}+ \|w\|^{\frac{q-s}{qs}}_W}
\right)\|w\|_W^{\frac{q+1}{q}} - c_2\lambda\|w\|_{W}^{r+1} - c_3\|w\|_W^{p+1}.
\end{aligned}
\end{equation}

Next, consider the maximization problem
$$
\sup \limits_{R \geq 0}\left( c_1R^{\frac{q+1}{q}} - c_3 R^{p+1}\right).
$$

Since $qp>1$, there exists a unique $R_0>0$ such that
$$
c_1R^{\frac{q+1}{q}}_0 - c_3 R_0^{p+1}=\sup \limits_{R \geq 0}\left( c_1R^{\frac{q+1}{q}} - c_3 R^{p+1}\right)
$$
and $c_1R^{\frac{q+1}{q}}_0 - c_3 R_0^{p+1}>0$. Observe also that 
\begin{equation}\label{eq:deriv_R0}
c_1\frac{q+1}{q}R_0^{\frac{1}{q}}= c_3(p+1)R_0^{p}.
\end{equation}

Fix any $\eta \in (0,1)$ and set $r_0:=\eta R_0$. Using that {$qp>1$} we notice that $\eta^{\frac{1}{q}}\geq \eta^p$ and from  \eqref{eq:deriv_R0}  for any $R\in [r_0,R_0)$,
$$c_1R^{\frac{q+1}{q}} - c_3 R^{p+1}>0 \quad \hbox{and} \quad  c_1\frac{q+1}{q}R^{\frac{1}{q}}> c_3(p+1)R^{p}.
$$

Set $\delta_0$, $\mu_0,\lambda_0>0$ as follows
\begin{equation}\label{mulambdazero}
\delta_0:= \frac{1}{3}\frac{qp-1}{q(p+1)}, \qquad {\mu_0:= \frac{\delta_0^s \, R_0^{\frac{q-s}{q}}\eta^{\frac{s(q+1)}{q}}}{(1-\delta_0 \eta^{\frac{q+1}{q}})^s}}, \qquad \lambda_0:= \frac{\delta_0 c_1}{c_2}\eta^{\frac{q+1}{q}}R_0^{\frac{1}{q} -r}.
\end{equation}

{Observe that $\delta_0\in (0,1)$ and $\eta^{\frac{q+1}{q}}< \frac{1}{\delta_0}$ so that $\mu_0$ is well-defined.} 

\vskip 3pt
From \eqref{ineq:lowerboundenerg} for any $\mu \in [0, \mu_0]$, any $\lambda \in[0,\lambda_0]$ and any ${w}\in W$ with $r_0\leq \|w\|_W\leq R_0$, we set $\|w\|_W=R$ and we estimate $$
\begin{aligned}
J_{\mu,\lambda}(w) \geq &
c_1 R^{\frac{q+1}{q}} - c_2 \lambda R^{r+1} - c_3 R^{p+1} -\frac{c_1\mu^{\frac{1}{s}}R^{\frac{q+1}{q}}}{\mu^{\frac{1}{s}}+ R^{\frac{q-s}{qs}}}
\\
{>}& c_1\left(1 - \frac{q+1}{q}\frac{1}{p+1}\right)R^{\frac{q+1}{q}} -c_2 \lambda R^{r+1}  -\frac{c_1\mu^{\frac{1}{s}}R^{\frac{q+1}{q}}}{\mu^{\frac{1}{s}}+ R^{\frac{q-s}{qs}}}
\\
\geq & R^{\frac{q+1}{q}}\Big(c_1 \frac{qp-1}{q(p+1)} - c_2 \lambda_0 R^{r-\frac{1}{q}} \Big)-\frac{c_1\mu^{\frac{1}{s}}R^{\frac{q+1}{q}}}{\mu^{\frac{1}{s}}+ R^{\frac{q-s}{qs}}}.
\end{aligned}
$$

From {\eqref{mulambdazero}} and setting
$$
c_0:=\frac{1}{3}c_1 \frac{qp-1}{q(p+1)}r_0^{\frac{q+1}{q}}>0,
$$
we find that for any ${w}\in W$ with $r_0\leq \|{w}\|_W\leq R_0$,
$$
\begin{aligned}
J_{\mu,\lambda}(w) \geq & 2c_0  -\frac{c_1\mu^{\frac{1}{s}}R_0^{\frac{q+1}{q}}}{\mu^{\frac{1}{s}}+ R_0^{\frac{q-s}{qs}}}\geq  c_0>0.  \qquad (\hbox{since }0\leq\mu\leq\mu_0)
\end{aligned}
$$

This completes the proof of the lemma.
\end{proof}

\begin{lemma}\label{negativeenergy}
Let $R_0>0$ and $\lambda_0>0$ be as in part {\it (ii)} of Lemma \ref{positiveenergyradius}. There exists $\lambda'_0\in(0,\lambda_0)$ depending only on $q,r$, such that given any $\mu \geq 0$ and $\lambda\in (0,\lambda_{0}']$, there exists $v_{\lambda}\in W\setminus\{0\}$ satisfying that 
\begin{itemize}
\item[(i)] $\|v_{\lambda}\|_W\leq R_0$ and $\|v_{\lambda}\|_W \to 0$ as $\lambda \to 0^+$ and

\item[(ii)] $J_{\mu,\lambda}(v_{\lambda}) \leq  - \frac{1}{p+1}\|v_{\lambda}\|_{L^{p+1}(\Omega)}^{p+1}dx <0$.
\end{itemize}
\end{lemma}

\begin{proof}
Let $w\in W$, {$w\geq 0$ a.e. in $\Omega$ and} with $\|w\|_W=1$ be fixed. Using hypothesis {\rm (A3)} we also fix 
\begin{equation}\label{eq:sigma_choice}
\sigma >\frac{q}{1 -qr}
\end{equation}
and  
$\lambda_0'\in\Big(0,\min\big(\lambda_0, R_0^{\frac{1}{\sigma}}\big)\Big)$. Let $\mu\geq 0$ and $\lambda 
\in(0,\lambda_0']$ and set $v_{\lambda}:= \lambda^\sigma w \in W$. Notice that for $\lambda\in (0,\lambda'_0]$, 
$$
R_0 \geq \lambda^{\sigma}= \|v_{\lambda}\|_W \to 0, \quad \hbox{ as } \lambda \to 0
$$ 
and thus proving $(i)$.

To prove $(ii)$, we estimate
$$
\begin{aligned}
J_{\mu,\lambda}(v_{\lambda}) \leq & \frac{q}{q+1}\int_{\Omega}\psi(\mu,\Delta v_{\lambda})\Delta v_{\lambda}dx - \frac{\lambda}{r+1}\int_{\Omega}|v_{\lambda}|^{r+1} dx \\
& \hspace*{5cm} - \frac{1}{p+1}|v_{\lambda}|^{p+1}dx \qquad (\hbox{from} \quad \eqref{ComparisonInequality})\\
\leq & \frac{q}{q+1}\int_{\Omega}|\Delta v_{\lambda}|^{\frac{q+1}{q}}dx - \frac{\lambda}{r+1}\int_{\Omega}|v_{\lambda}|^{r+1} dx\\
&\hspace*{5cm} - \frac{1}{p+1}|v_{\lambda}|^{p+1}dx\qquad (\hbox{from} \quad \eqref{growthpsi})\\
=&\underbrace{\frac{q}{q+1}\lambda^{\frac{\sigma(q+1)}{q}}\|w\|_{W}^{\frac{q+1}{q}} - \frac{\lambda^{1+ \sigma(r+1)}}{r+1}\|w\|_{L^{r+1}(\Omega)}^{r+1}}_{A_{\lambda}} - \frac{1}{p+1}\|v_{\lambda}\|_{L^{p+1}}^{p+1}.
\end{aligned} 
$$

The fact that $\|w\|_W=1$ yields 
$$
A_{\lambda}\leq  \frac{q}{q+1}\lambda^{\frac{\sigma(q+1)}{q}} - \frac{\lambda^{1+ \sigma(r+1)}}{r+1}\|w\|^{r+1}_{L^{r+1}(\Omega)}.
$$

{Using \eqref{eq:sigma_choice}, we may choose} $\lambda_0'>0$ smaller if necessary so that 
\begin{equation*}\label{eq:choice_lambda0}
\big(\lambda_0'\big)^{\frac{\sigma(1-qr)}{q}{-1}}\leq \frac{q+1}{q(r+1)}\|w\|_{L^{r+1}(\Omega)}^{r+1},
\end{equation*}
we find that $A_{\lambda}\leq 0$ for any $\lambda \in (0,\lambda_0']$ and consequently $J_{\mu,\lambda}(v_{\lambda})\leq - \frac{1}{p+1}\|v_{\lambda}\|_{L^{p+1}}^{p+1}<0$. This proves $(ii)$ and completes the proof of the lemma.
\end{proof}

{Let $R_0,\mu_0,\lambda_0$ be as described in Lemma \ref{positiveenergyradius} and let $\lambda_0'\in (0,\lambda_0)$ be as described in Lemma \ref{negativeenergy}}. 


\vskip 3pt
Consider the minimization problem
\begin{equation}\label{minproblemJmulambda}
{{\sf m}_{\mu,\lambda}:=}\inf \limits_{\substack{w\in W,\\ \|w\|_{W}\leq R_0}}J_{\mu,\lambda}(w)
\end{equation}
for $\mu\in [0,\mu_0]$ and $\lambda\in (0,{\lambda}_0']$. {From part $(i)$ in Lemma \ref{positiveenergyradius}, ${\sf m}_{\mu,\lambda}$ is well defined.}

\begin{prop}\label{slnminproblemJmulambda}
For any $\mu\in[0,\mu_0]$ and any $\lambda \in (0, {\lambda}_0']$, there exists ${\sf w}_{\mu,\lambda}\in W$ a non-negative weak solution of \eqref{4thorderBVP}-\eqref{eqn:navier_bd_cond} {such that $0<\|{\sf w}_{\mu,\lambda}\|_W<R_0$ and  $J_{\mu,\lambda}({\sf w}_{\mu,\lambda})={\sf m}_{\mu,\lambda}$.}
\end{prop}

\begin{proof} First, recall that the functional 
$$
W\ni w \mapsto \int_{\Omega}\Psi(\mu,\Delta w)dx \in [0,\infty)
$$
is weakly lower semicontinuous. On the other hand, by virtue of the Sobolev embeddings \eqref{SobolevEmbedding_q} and Vainberg's Lemma (see \cite{VainbergBook}) the function 
$$
W\ni w\mapsto \int_{\Omega}{F}_+(\lambda,w)dx\in \R
$$
is sequentially weakly continuous. Thus, we conclude that $J_{\mu,\lambda}$ is {sequentially} weakly lower semicontinuous in $W$.

\vskip 6pt
From {Theorem 1.2, chapter I} in \cite{StruweBook2008} (see also Corollary 3.2  in \cite{BrezisBook2011}), for any $\mu \in [0,\mu_0]$ and any $\lambda \in (0,{\lambda}'_0]$ there exists ${\sf w}_{\mu,\lambda}\in W$ such that $\|{\sf w}_{\mu,\lambda}\|_W \leq R_0$ and
\begin{equation*}\label{minimizationJmulambda}
J_{\mu,\lambda}({\sf w}_{\mu,\lambda}) = \min \limits_{\substack{v\in W,\\ \|v\|_{W}\leq R_0}}J_{\mu,\lambda}(v)={\sf m}_{\mu,\lambda}.
\end{equation*}

Next, we observe from Lemma \ref{negativeenergy} that $J_{\mu,\lambda}({\sf w}_{\mu,\lambda})={\sf m}_{\mu,\lambda}<0$. The previous remark, together with part $(ii)$ in Lemma \ref{positiveenergyradius}, implies that $\|{\sf w}_{\mu,\lambda}\|_{W}<R_0$.

\vskip 6pt
Since $J_{\mu,\lambda}\in C^1(W)$,  ${\sf w}_{\mu,\lambda}$ is a critical point of $J_{\mu,\lambda}$, i.e., ${\sf w}_{\mu,\lambda}\in W$ is a weak solution of \eqref{4thorderBVP}-\eqref{eqn:navier_bd_cond}. Theorem \ref{theo:regularity_4thorder} and the maximum principle yields that ${\sf w}_{\mu,\lambda} > 0$ a.e. in $\Omega$.
\end{proof}

\vskip 3pt

Next, we find the second non-negative weak solution of \eqref{4thorderBVP}-\eqref{eqn:navier_bd_cond}, by applying the {\it Mountain Pass Theorem} (see Theorem 2.1 in {\cite{AMBROSETTIRABINOWITZ1973} and Theorem 6.1 from Chapter II in \cite{StruweBook2008}}) to the energy $J_{\mu,\lambda}$.

\begin{prop}\label{2ndnonnegativesln}
For any $\mu \in [0, \mu_0]$ and any $\lambda \in (0, {\lambda}_0']$, there exists ${\sf v}_{\mu,\lambda}\in W$ a non-negative weak solution of \eqref{4thorderBVP}-\eqref{eqn:navier_bd_cond} of mountain pass type.
\end{prop}

\begin{proof}
 First, observe that for any $\varphi \in W\setminus \{0\}$ with $\varphi(x) \geq 0$ a.e. $x\in \Omega$ and any $t >0$, 
$$
\begin{aligned}
J_{\mu,\lambda}( t\varphi) \leq & \int_{\Omega}\psi(\mu,t\Delta \varphi)t\Delta \varphi dx - \frac{\lambda t^{r+1}}{r+1}\int_{\Omega}\varphi^{r+1} dx \\
& \hspace*{5cm}- \frac{t^{p+1}}{p+1}\int_{\Omega}\varphi^{p+1}dx \qquad \hbox{from} \quad \eqref{ComparisonInequality}\\
\leq & t^{\frac{q+1}{q}}\int_{\Omega}|\Delta \varphi|^{\frac{q+1}{q}}dx- \frac{\lambda t^{r+1}}{r+1}\int_{\Omega}\varphi^{r+1} dx\\
& \hspace*{5cm} - \frac{t^{p+1}}{p+1}\int_{\Omega}\varphi^{p+1}dx \qquad \hbox{from} \quad \eqref{growthpsi}{,}
\end{aligned} 
$$
and since {$qp>1$}, it follows from $(A2)$ that  
$$
\lim \limits_{t \to +\infty} J_{\mu,\lambda}(t\varphi) = - \infty.
$$

Let ${\sf w}_{\mu,\lambda}\in W$ be the solution to \eqref{minproblemJmulambda} predicted in Proposition \ref{slnminproblemJmulambda}. Consider also ${\tilde{\varphi}}\in W$ with ${\tilde{\varphi}} \geq 0$ a.e. in $\Omega$ and  $\|{\tilde{\varphi}}\|_W$ large enough so that $J_{\mu,\lambda}({\tilde{\varphi}})<{\sf m}_{\mu,\lambda}=J_{\mu,\lambda}({\sf w}_{\mu,\lambda})$. Set
$$
{\sf c}_{\mu,\lambda}:=\inf \limits_{\gamma\in \Gamma}\max \limits_{t\in[0,1]}J_{\mu,\lambda}(\gamma(t)),
$$
where $\Gamma:=\{\gamma\in C([0,1],W)\,:\,\gamma(0)={\sf w}_{\mu,\lambda} \quad \hbox{and} \quad \gamma(1)={\tilde{\varphi}}\}$.

\vskip 6pt
Lemmas \ref{positiveenergyradius} and \ref{negativeenergy} guarantee that ${\sf c}_{\mu,\lambda}$ is well defined and 
$$
J_{\mu,\lambda}({\tilde{\varphi}})< J_{\mu,\lambda}({\sf w}_{\mu,\lambda})< 0<\inf \limits_{\substack{v\in W\\ \|v\|_{W}=R_0}}J_{\mu,\lambda}(v)
\leq {\sf c}_{\mu,\lambda}.
$$

{Since $J_{\mu,\lambda}$ satisfies the $(PS)-$condition, Theorem 2.1 in \cite{AMBROSETTIRABINOWITZ1973} (see also Theorem 6.1 in \cite{StruweBook2008})}, yields the existence of ${\sf v}_{\mu,\lambda}\in W\setminus\{0\}$ a weak solution of \eqref{4thorderBVP} with $J_{\mu,\lambda}({\sf v}_{\mu,\lambda})={\sf c}_{\mu,\lambda}$. 

\vskip 6pt
{Theorem \ref{theo:regularity_4thorder} and the maximum principle guarantees that ${\sf v}_{\mu,\lambda} > 0$ in $\Omega$.} This completes the proof of the proposition.
\end{proof}

\begin{remark}
Since $\mu\in [0,\mu_0]$ in Propositions \ref{slnminproblemJmulambda} and \ref{2ndnonnegativesln}, these two results extend part {\rm III} from Theorem 1.3 in \cite{dosSantos2009-2} (which, according to our notations, deals with the case $\mu=0$).    
\end{remark}

\section{Proof of Theorem \ref{theo:bifur_diagram}}\label{sect:proof_theo_bif}

In this part we prove Theorem \ref{theo:bifur_diagram}. Recall that we assume {\rm (A1)}, {\rm (A2)} and {\rm (A3)} and $qp>1$. 

\vskip 3pt
Consider the fourth order BVP
 \begin{equation}\label{4thorderBVP_sublinear}
\Delta\Big(|\Delta \omega|^{\frac{1}{q}-1}\Delta \omega\Big) = \omega_+^{r} \quad \hbox{in} \quad \Omega, \qquad \omega=\Delta \omega =0 \quad \hbox{on} \quad \partial \Omega.
\end{equation}

\vskip 3pt
A straightforward minimization argument yields the existence of a nontrivial weak solution $\omega\in W$ of \eqref{4thorderBVP_sublinear}. 

\vskip 3pt
Set ${\rm u}:=-|\Delta \omega|^{\frac{1}{q}-1}\Delta \omega$ and ${\rm v}:=\omega$. Observe that ${\rm u}\in L^{q+1}(\Omega)$. Also, ${\rm v}\in W$ and from {\rm (A1)} we find that ${\rm v}\in L^{p+1}(\Omega)$. Even more, the pair $({\rm u},{\rm v})$ is a nontrivial very weak solution of the system 
$$
-\Delta {\rm u}= {\rm v}_+^r, \quad -\Delta {\rm v}=|{\rm u}|^{q-1}{\rm u} \quad \hbox{in} \quad \Omega,
$$ 
with ${\rm u}={\rm v}=0$ on  $\partial \Omega$.

\vskip 3pt 
Theorem \ref{theo:Theorem_1} yields that ${\rm u},{\rm v}\in C^2_0(\overline{\Omega})$, so that $({\rm u},{\rm v})$ is a classical solution of the system. On the other hand, from the maximum principle ${\rm u},{\rm v}>0$ in $\Omega$. 

\vskip 3pt
Let $\underline{\lambda}\geq 0$ and write $\underline{u}:=\underline{\lambda}^{\frac{1}{1-qr}}{\rm u}$ and $\underline{v}=\underline{\lambda}^{\frac{q}{1-qr}}{\rm v}$. A direct calculation shows that the pair $(\underline{u},\underline{v})$ is a classical solution of the system
\begin{equation}\label{eqn:aux_system_sublinear}
-\Delta \underline{u}=\underline{\lambda}\, \underline{v}^r, \quad -\Delta \underline{v}=\underline{u}^q \quad \hbox{in} \quad \Omega, \qquad \underline{u}=\underline{v}=0 \quad \hbox{on} \quad \partial \Omega.
\end{equation}

\begin{lemma}\label{lemma:theo1.4_I} Let $\mu \geq 0$ and $\lambda \geq \underline{\lambda}\geq 0$. For any $u,v\in C^2(\overline{\Omega})$ such that $(u,v)$ is a nontrivial solution of \eqref{HS1}, 
$$
u > \underline{u} \quad  \hbox{and} \quad v> \underline{v}\quad \hbox{in} \quad \Omega. 
$$
\end{lemma}

\begin{proof}
From \eqref{HS1}, in $\Omega$ it holds that
$$
\begin{aligned}
-\Delta u &=\lambda v^r + v^p > \lambda v^r \geq  \underline{\lambda}v^r&\\
-\Delta v &=\mu u^s + u^q \geq u^q.&
\end{aligned}
$$

Lemma 4.2 in \cite{AGUDELORUFVELEZ2023} {(noticing that the proof of this lemma uses only that $qr<1$)} yields that $u\geq \underline{u}$ and $v\geq \underline{v}$ in $\Omega$. Thus, we estimate again using \eqref{HS1} and \eqref{eqn:aux_system_sublinear} to find that 
$$
-\Delta u > -\Delta \underline{u} \quad \hbox{and} \quad -\Delta v \geq -\Delta \underline{v} \quad \hbox{in}\quad \Omega.
$$ 

The {\it Strong Maximum Principle} yields the strict inequalities and this proves the lemma.
\end{proof}

\vskip 3pt

Next, let $\overline{\lambda}>0$ and $\overline{\mu}
\geq 0$ be such that $( \overline{\lambda},\overline{\mu})\in S$. Let $\overline{w}\in W$ be a nontrivial solution of \eqref{4thorderBVP}-\eqref{eqn:navier_bd_cond} for $ \overline{\lambda}$ and  $\overline{\mu}$. Write $\overline{u}=-\psi(\overline{\mu},\Delta \overline{w})$ and $\overline{v}=\overline{w}$ and observe that $\overline{u}\in \widetilde{W}$ and $\overline{v}\in W$.

\vskip 3pt
Let $\lambda \in (0,\overline{\lambda}]$ and $\mu\in[0,\overline{\mu}]$ be arbitrary. Fix $\underline{\lambda}\in (0,\lambda)$ and let $\underline{u}$ and $\underline{v}$ be as in \eqref{eqn:aux_system_sublinear}. From Lemma \ref{lemma:theo1.4_I}, we have that $\overline{v}>\underline{v}$ in $\Omega$.

\vskip 3pt

Consider the truncated energy,
$$
\bar{J}_{\mu,\lambda}(w):= \int_{\Omega}\Big(\Psi(\mu,\Delta w)-\bar{F}(\lambda,x,w)\Big)dx \quad \hbox{for}\quad w\in W, 
$$
where $\bar{F}(\lambda,x,\zeta):=\int_{0}^\zeta \bar{f}(\lambda,x,t)dt$ and
$$
\bar{f}(\lambda,x,\zeta):=\left\{
\begin{aligned}
f_+(\lambda,\overline{v}(x)),& \quad \hbox{if} \quad  \zeta \geq \overline{v}(x);\\
f_+(\lambda,\zeta),& \quad \hbox{if} \quad  \underline{v}(x)\leq \zeta \leq \overline{v}(x);\\
f_+(\lambda,\underline{v}(x)),& \quad \hbox{if} \quad  \zeta \leq \underline{v}(x).
\end{aligned}
\right.
$$

Observe that $\bar{f}$ is H\"{o}lder continuous and uniformly bounded. It is standard to verify that $\bar{J}_{\mu,\lambda}\in C^1(W)$ with 
$$
D\bar{J}_{\mu,\lambda}(w)\varphi=\int_{\Omega}\Big(\psi(\mu,\Delta w)\Delta \varphi - \bar{f}(\lambda,x,w)\varphi\Big)dx \quad \hbox{for}\quad \varphi \in W.
$$ 

\begin{remark}
The idea of studying truncated energy functionals is well known (see e.g. \cite{DeFigueiredo1987} and references there in).
\end{remark}

\begin{lemma}\label{lemma:theo1.4_II}
Let $w\in W$ be a critical point of $\bar{J}_{\mu,\lambda}$ in $W$. Then, $w, \psi(\mu,\Delta w)\in C^2_0(\overline{\Omega})$ and $w$ is also a critical  point of ${J}_{\mu,\lambda}$. \end{lemma} 

\begin{proof} Set $u=-\psi(\mu,\Delta w)$ and $v=w$. it is directly verified that the pair $(u,v)$ is a very weak solution of the system 
\begin{equation}\label{eqn:system_trunc_bd}
-\Delta u=\bar{f}(\lambda,x,v), \qquad -\Delta v =g(\mu,u) \quad \hbox{in }\Omega
\end{equation}
with $u=v=0$ on $\partial \Omega$, where we recall that $g(\mu,\zeta)=\mu \,|\zeta|^{s-1}\zeta+|\zeta|^{q-1}\zeta$ for $\zeta \in \R$. From Theorem \ref{theo:Theorem_1}, $u,v\in C^2_0(\overline{\Omega})$.

\vskip 3pt
Next, using that $\lambda\in (\underline{\lambda},\overline{\lambda}]$, it follows directly from the definition of $\bar{f}$ and since $v=w$,
$$
f_+(\underline{\lambda},\underline{v})< f_+(\lambda,\underline{v})\leq  \bar{f}(\lambda,x,v)\leq f_+({\lambda},\overline{v})\leq  f_+(\overline{\lambda},\overline{v}) \quad \hbox{in }\Omega.
$$

From the first equation in \eqref{eqn:system_trunc_bd}$$
-\Delta \underline{u} <  -\Delta u \leq -\Delta \overline{u} \quad \hbox{in }\Omega.
$$

The comparison principle yields that $\underline{u}<u\leq\overline{u}$ in $\Omega$. Even more, either $u<\overline{u}$ in $\Omega$ or  $u=\overline{u}$ in $\Omega$. Consequently,
$$
g(0,\underline{u})< g(\mu,u)\leq g(\overline{\mu},\overline{u}) \quad \hbox{in }\Omega.
$$

Using \eqref{eqn:aux_system_sublinear} and the second equation in \eqref{eqn:system_trunc_bd}, $$
-\Delta \underline{v} < -\Delta v \leq  -\Delta \overline{v} \quad \hbox{in }\Omega.
$$

The comparison principle and the fact that $v=w$ yield that $\underline{v}<w\leq \overline{v}$ in $\Omega$. Observe also that in the case $u<\overline{u}$, then $v<\overline{v}$ in $\Omega$. In either case $w$ is a solution of \eqref{4thorderBVP}-\eqref{eqn:navier_bd_cond}. This proves the lemma.
\end{proof}

From the previous proof we observe that in the case that $\lambda \in (\underline{\lambda},\overline{\lambda})$ and $\mu \in [0,\overline{\mu}]$, then
\begin{equation}\label{eqn:theo1.4_strictinq}
\underline{u}<u<\overline{u} \quad \hbox{and} \quad \underline{v}<v<\overline{v}  \qquad \hbox{in} \quad \Omega.
\end{equation}

The next three lemmas show the existence of a nontrivial local minimum of $J_{\mu,\lambda}$ in $W$. Set
$$
(\underline{v},\overline{v}):=\big\{\varphi \in C^2_0(\overline{\Omega})\,:\, \underline{v}<\varphi<\overline{v} \quad \hbox{in }\Omega\big\}.
$$

\begin{lemma}\label{lemma:theo1.4_IV}
There exists ${w}_0\in C^2_0(\overline{\Omega})\setminus \{0\}$, a global minimum for $\bar{J}_{\mu,\lambda}$ in $W$ and a critical point of $J_{\mu,\lambda}$ in $W$. Moreover, if $\lambda \in (\underline{\lambda},\overline{\lambda})$ and $\mu\in [0,\overline{\mu}]$, then $w_0\in (\underline{v}, \overline{v})$.
\end{lemma}

\begin{proof} First observe that $\bar{J}_{\mu,\lambda}$ is coercive in $W$ and thus bounded from below. It is also readily verified that $\bar{J}_{\mu,\lambda}$ is sequentially weakly  lower semicontinuous in $W$. Applying {Theorem 1.2, chapter I} in \cite{StruweBook2008} (see also Corollary 3.2  in \cite{BrezisBook2011})
to $\bar{J}_{\mu,\lambda}$ in $W$ we find a global minimum $w_0\in W$ for $\bar{J}_{\mu,\lambda}$. Finally, from Lemma \ref{inequalitiespsi} and \eqref{eqn:aux_system_sublinear}, 
$$
\begin{aligned}
\bar{J}_{\mu,\lambda}(\underline{v})\leq & \frac{q}{q+1}\int_{\Omega}|\Delta \underline{v}|^{\frac{q+1}{q}}dx - \int_{\Omega}\overline{F}(\lambda,x,\underline{v})dx\\
=&\frac{q \,\underline{\lambda}}{q+1}\int_{\Omega}\underline{v}^{r+1}dx-\int_{\Omega}\big(\lambda \underline{v}^{r+1}+ \underline{v}^{p+1}\big)dx\\
\leq& {- \frac{\lambda}{q+1}\int_{\Omega}\underline{v}^{r+1}dx - \int_{\Omega}\underline{v}^{p+1}dx<0}.
\end{aligned}
$$ 

Thus, $\bar{J}_{\mu,\lambda}(w_0)<0$ and hence $w_0\neq 0$ in $W$. From Lemma \ref{lemma:theo1.4_II} we obtain that $w_0\in C^2_0(\overline{\Omega})$. The second part follows as in \eqref{eqn:theo1.4_strictinq}. \end{proof}

\begin{lemma}\label{lemma:theo1.4_III}
For $w\in W$ with $\underline{v}\leq w\leq \overline{v}$ a.e. in $\Omega$,
\begin{equation}\label{eqn:J_barJ}
\bar{J}_{\mu,\lambda}(w)=J_{\mu,\lambda}(w) \underbrace{-\frac{\lambda\,r}{r+1}\|\underline{v}\|^{r+1}_{L^{r+1}(\Omega)}-\frac{p}{p+1}\|\underline{v}\|^{p+1}_{L^{p+1}(\Omega)}}_{\hbox{constant}}.
\end{equation} 
\end{lemma}

\begin{proof} Observe that for any $w\in W$,
$$
\bar{J}_{\mu,\lambda}(w)=J_{\mu,\lambda}(w) + \int_{\Omega}I_{\lambda,w}(x)dx,
$$
where 
$$
I_{\lambda,w}(x):=F_+(\lambda,w)-\bar{F}(\lambda,x,w) \quad \hbox{for a.e. }x\in \Omega. 
$$

If $w\in W$ with $\underline{v}\leq w\leq \overline{v}$ a.e in $\Omega$,
$$
\begin{aligned}
I_{\lambda,w}(x)&=\int_{0}^{w(x)}\big(f_+(\lambda,\zeta) -\bar{f}(\lambda,x,\zeta)\big)d\zeta\\
&=\int_{0}^{\underline{v}(x)}\cdots \,d\zeta + \underbrace{\int_{\underline{v}(x)}^{w(x)}\cdots \,d\zeta}_{=0}\\
&= \int_{0}^{\underline{v}(x)}\Big(f_+(\lambda,\zeta) -f_+(\lambda,\underline{v}(x))\Big)d\zeta\\
&=-\frac{\lambda\,r}{r+1}\underline{v}^{r+1}(x)-\frac{p}{p+1}\underline{v}^{p+1}(x).
\end{aligned}
$$

Integrating over $\Omega$,  equality \eqref{eqn:J_barJ} follows and this proves the lemma. \end{proof}

\begin{lemma}\label{lemma:theo1.4_V}
Let $\lambda \in (\underline{\lambda},\overline{\lambda})$ and $\mu\in [0,\overline{\mu}]$ and let $w_0\in C_0^2(\overline{\Omega})$ be a global minimum in $W$ of $\bar{J}_{\mu,\lambda}$. Then, $w_0$ is a local minimum in $W$ of ${J}_{\mu,\lambda}$ .
\end{lemma}

\begin{proof}
First notice that $w_0\in (\underline{v},\overline{v})$. Write $u_0=-\psi(\mu,\Delta w_0)$ and $v_0=w_0$. Observe that $u_0,v_0 \in C^2_0(\overline{\Omega})$ and the pair $(u_0,v_0)$ is a classical solution of \eqref{eqn:system_trunc_bd} with $u_0,v_0>0$ in $\Omega$. 

\vskip 3pt
Using Hopf's Lemma we find $\delta>0$ small such that 
$$
B_{\delta}^{C^2_0}(w_0):= \big\{\varphi\in C^2_0(\overline{\Omega})\,:\, \|w_0-\varphi\|_{C^2(\overline{\Omega})}<\delta\big\}
\subset (\underline{v},\overline{v}).
$$

From \eqref{eqn:J_barJ}, $w_0$ is a local minimum of ${J}_{\mu,\lambda}$ in  $B_{\delta}^{C^2_0}(w_0)$. The conclusion of the lemma follows from Theorem \ref{theo:local_min}.
\end{proof}

Next, we show the existence of a second nontrivial critical point of $J_{\mu,\lambda}$. In this part, we assume $\mu\in [0,\overline{\mu}]$ and $\lambda\in (\underline{\lambda},\overline{\lambda})$.

\vskip 3pt
Consider the truncated energy $\hat{J}_{\mu,\lambda}:W\to \R$ defined by
$$
\hat{J}_{\mu,\lambda}(w):=\int_{\Omega}\Psi(\mu,\Delta w)dx -\int_{\Omega}\hat{F}(\lambda,x,w)dx \quad \hbox{for }w\in W,
$$
where $\hat{F}(\lambda,x,\zeta):=\int_{0}^\zeta\hat{f}(\lambda,x,t)dt$ and
$$
\hat{f}(\lambda,x,\zeta):=\left\{
\begin{aligned}
f_+(\lambda,\underline{v}(x)),& \quad \hbox{if} \quad  \zeta \leq \underline{v}(x)\\
f_+(\lambda,\zeta),& \quad \hbox{if} \quad  \underline{v}(x)\leq \zeta .
\end{aligned}
\right.
$$

Proceeding as we did above for the functional $\bar{J}_{\mu,\lambda}$, the following properties are verified:
\begin{itemize}
\item[I.] $\hat{J}_{\mu,\lambda}\in C^1(W)$ with 
$$
D\hat{J}_{\mu,\lambda}(w)\varphi=\int_{\Omega}\Big(\psi(\mu,\Delta w)\Delta \varphi - \hat{f}(\lambda,x,w)\varphi\Big)dx \quad \hbox{for}\quad \varphi \in W.
$$ 

\item[II.] If $w\in W$ is such that $D\hat{J}_{\mu,\lambda}(w)\equiv0$ in $W^*$, then $w,\psi(\mu,\Delta w)\in C^2_0(\overline{\Omega})$ and $w\geq \underline{v}$ in $\Omega$. In particular, $D{J}_{\mu,\lambda}(w)\equiv 0$ in $W^*$. 

\item[III.] For any $w\in W$ with $w \geq \underline{v}$ a.e. in $\Omega$,
$$
\hat{J}_{\mu,\lambda}(w)=J_{\mu,\lambda}(w) \underbrace{-\frac{\lambda\,r}{r+1}\|\underline{v}\|^{r+1}_{L^{r+1}(\Omega)}-\frac{p}{p+1}\|\underline{v}\|^{p+1}_{L^{p+1}(\Omega)}}_{\hbox{constant}}.
$$

\item[IV.] Since $\lambda>\underline{\lambda}>0$, then $w_0\in C^2_0(\overline{\Omega})$ (established in Lemmas \ref{lemma:theo1.4_IV}-\ref{lemma:theo1.4_V})  is a also local minimum in $W$ of $\hat{J}_{\mu,\lambda}$.

\vskip 3pt
\item[V.] Since $\underline{v}\in C^2_0(\overline{\Omega})$ and $\underline{v}>0$ in $\Omega$,  
$$
\lim \limits_{t\to \infty} \hat{J}_{\mu,\lambda}(t \underline{v})=-\infty.
$$
\end{itemize}

\begin{lemma}\label{lemma:theo1.4_VI}
Let $\lambda \in (\underline{\lambda},\overline{\lambda})$ and $\mu\in [0,\overline{\mu}]$. There exists $w_1\in C^2_0(\overline{\Omega})\setminus\{w_0\}$ a nontrivial critical point in $W$ of $\hat{J}_{\mu,\lambda}$, which is also a critical point in $W$ of $J_{\mu,\lambda}$.
\end{lemma}

\begin{proof}
By virtue of hypothesis (A1), $\hat{J}$ satisfies the $(PS)-$condition. On the other hand, using IV., if $w_0$ is not an isolated critical point of $\hat{J}_{\mu,\lambda}$, then there are infinitely many critical points and a nontrivial second one can be selected. Otherwise, $w_0$ is an isolated critical point of $\hat{J}_{\mu,\lambda,}$ and Proposition 5.42 in \cite{MotrMotrPapa2014} (see also {Theorem 1 in \cite{PUCCISERRIN1984}}) yields the existence of  another critical point $w_1\in W$ with $\hat{J}_{\mu,\lambda}(w_1)>\hat{J}_{\mu,\lambda}(w_0)$.

\vskip 3pt
Since, 
$$
D\hat{J}_{\mu,\lambda}(0)\varphi = -\int_{\Omega}f_+(\lambda,\underline{v}(x))\varphi dx \quad \hbox{for }\varphi\in W,
$$
then $w_1\neq 0$ in $W$. From II., $w_1\in C^2_0(\overline{\Omega})$ is a critical point of $J_{\mu,\lambda}$. This concludes the proof.
\end{proof}

\begin{lemma}\label{lemma:theo1.4_VII}
Let $\mu\geq 0$ and $ \widehat{\lambda}>0$. Assume that $\{\widehat{\lambda}_k\}_{k}\subset (0,\infty)$ is a strictly increasing sequence such that for any $k\in \mathbb{N}$, $(\widehat{\lambda}_k,\mu)\in S$ and $\widehat{\lambda}_k\to \widehat{\lambda}$. Then, $(\widehat{\lambda},\mu)\in S$. 
\end{lemma}

\begin{proof}
Let $\mu\geq 0$ and for the sake of convenience we index the aforementioned sequence as $\{\widehat{\lambda}_{2m-1}\}_{m\in \mathbb{N}}$ with  $\widehat{\lambda}_{2m-1}<\widehat{\lambda}_{2m+1}$ and $\widehat{\lambda}_{2m-1}\to \widehat{\lambda}$ as $m\to \infty$.

\vskip 3pt
For each $m\in \mathbb{N}$, let $w_{2m-1}\in C^2_0(\overline{\Omega})$ be a weak solution to \eqref{4thorderBVP}-\eqref{eqn:navier_bd_cond} with $\lambda=\widehat{\lambda}_{2m-1}$ and let also $\omega\in C^2_0(\overline{\Omega})$ be the weak solution of \eqref{4thorderBVP_sublinear}.  

\vskip 3pt
For each $m\in \mathbb{N}$, select $\widehat{\lambda}_{2m}\in (\widehat{\lambda}_{2m-1},\widehat{\lambda}_{2m+1})$ and set 
$$
\underline{\lambda}:=
\widehat{\lambda}_{2m-1}, \qquad \lambda=\widehat{\lambda}_{2m} \qquad \hbox{and} \qquad \overline{\lambda}:=
\widehat{\lambda}_{2m+1}. 
$$  

Write also 
$$
\underline{v}(x):=
\underline{\lambda}^{\frac{q}{1-qr}}\omega(x)\qquad \hbox{and}\qquad \overline{v}(x):=w_{2m+1}(x).
$$

Using that $\underline{\lambda} <\lambda<\overline{\lambda}$ and invoking Lemma \ref{lemma:theo1.4_I}, we find that $$
\underline{v}<\overline{v} \quad \hbox{in} \quad \Omega.
$$

From Lemmas \ref{lemma:theo1.4_II}, \ref{lemma:theo1.4_IV},   \ref{lemma:theo1.4_III} and \ref{lemma:theo1.4_V}, there exists $w_{2m}\in C^2_0(\overline{\Omega})$ a global minimum of $\bar{J}_{\mu,\widehat{\lambda}_{2m}}$ in $W$ with 
$$
\widehat{\lambda}^{\frac{q}{1-qr}}_{2m-1}\omega< w_{2m}
<w_{2m+1}\quad \hbox{in} \quad \Omega.
$$ 
 
Using \eqref{eqn:J_barJ}, we estimate
\begin{equation}\label{ineq:bounded_seqI}
\begin{aligned}
J_{\mu,\widehat{\lambda}_{2m}}(w_{2m})&= \bar{J}_{\mu,\widehat{\lambda}_{2m}}(w_{2m})+\frac{\widehat{\lambda}_{2m}\,r}{r+1}\|\underline{v}\|^{r+1}_{L^{r+1}(\Omega)}+\frac{p}{p+1}\|\underline{v}\|^{p+1}_{L^{p+1}(\Omega)}\\
&\leq \bar{J}_{\mu,\widehat{\lambda}_{2m}}(\widehat{\lambda}^{\frac{q}{1-qr}}_{1}\omega)+\frac{\widehat{\lambda}_{2m}\,r}{r+1}\|\underline{v}\|^{r+1}_{L^{r+1}(\Omega)}+\frac{p}{p+1}\|\underline{v}\|^{p+1}_{L^{p+1}(\Omega)}\\
&\leq \bar{J}_{\mu,\widehat{\lambda}_{1}}(\widehat{\lambda}^{\frac{q}{1-qr}}_{1}\omega)+\frac{\widehat{\lambda}_{2m}\,r}{r+1}\|\underline{v}\|^{r+1}_{L^{r+1}(\Omega)}+\frac{p}{p+1}\|\underline{v}\|^{p+1}_{L^{p+1}(\Omega)}
\end{aligned}
\end{equation}
and consequently, $\{J_{\mu,\widehat{\lambda}_{2m}}(w_{2m})\}_{m}$ is bounded from above. 

\vskip 3pt
On the other hand, using that $w_{2m}$ is a weak solution of \eqref{4thorderBVP}-\eqref{eqn:navier_bd_cond},

\begin{equation}\label{ineq:bounded_seqII}
\begin{aligned}
J_{\mu,\widehat{\lambda}_{2m}}(w_{2m}) &=\int_{\Omega}\Psi(\mu,\Delta w_{2m})dx -\int_{\Omega}F_+(\widehat{\lambda}_{2m},w_{2m})dx\\
&= \int_{\Omega}\Big(\Psi(\mu,\Delta w_{2m})-\frac{1}{p+1}\psi(\mu,\Delta w_{2m})\Delta w_{2m}\Big)dx \\
& \qquad \qquad \qquad -\frac{\widehat{\lambda}_{2m}(p-r)}{(p+1)(r+1)}\int_{\Omega}(w_{2m})^{r+1}_+dx.
\end{aligned}
\end{equation}

Using part {\it (iv)} from Lemma \ref{asymptoticpsizeroinfinity} and  {\it (v)} from Lemma \ref{asymptoticsPSI}, and proceeding as in the proof of Lemma \ref{PSbounded}, for some constants $\mathcal{K}_{q,p},\mathcal{K}>0$ independent of $m$,
\begin{equation}\label{ineq:bounded_seqIII}
J_{\mu,\widehat{\lambda}_{2m}}(w_{2m}) \geq {\mathcal{K}_{q,p}}\|w_{2m}\|_W^{\frac{q+1}{q}} -\frac{\widehat{\lambda}_{2m}S_{q,r}(p-r)}{(p+1)(r+1)}\|w_{2m}\|_W^{r+1} -\mathcal{K},
\end{equation}
where for some $\delta>0$ small
$$
\mathcal{K}_{q,p}=\frac{qp-1}{(q+1)(p+1)} - \delta\frac{qp+p+q+1}{(q+1)(p+1)}.
$$

Since $\widehat{\lambda}_{2m} \to \widehat{\lambda}$, it follows from \eqref{ineq:bounded_seqI}, \eqref{ineq:bounded_seqII} and \eqref{ineq:bounded_seqIII} that the sequence $\{w_{2m}\}_{m}$ is bounded in $W$. The above procedure yields that $\{w_{2m}\}_{m\in\mathbb{N}}$ is a bounded $PS-$sequence in $W$ for $J_{\mu,\widehat{\lambda}}$. 

\vskip 3pt
The results from Section 5 imply that, passing to a subsequence if necessary, $w_{2m}\to w_*$ strongly in $W$ and $w_*$ is a weak solution of \eqref{4thorderBVP}-\eqref{eqn:navier_bd_cond} for $\mu$ and $\widehat{\lambda}$. By virtue of (A1), $w_{2m}(x)\to w_*(x)$ for $x\in \Omega$. So that $w_*(x)\geq\widehat{\lambda}^{\frac{q}{1-qr}}\omega(x)>0$ for $x\in \Omega$. This completes the proof of the lemma.
\end{proof}

\begin{remark}
    At this point, we would like to summarize part of the content of the previous lemmas. If we draw the $\lambda \mu $-plane and locate the point $(\overline{\lambda}, \overline{\mu}) \in S$ in it,  then the rectangle $(0,\overline{\lambda}]\times [0,\overline{\mu}]$ is completely contained in $S$. This fact will be extensively used in the proof of our main theorem, which we now present.
\end{remark}

\underline{\it Proof of Theorem \ref{theo:bifur_diagram}}. From Propositions \ref{slnminproblemJmulambda} and \ref{2ndnonnegativesln},  $(0,\lambda_0')\times[0,\mu_0]\subset S$. Given $\mu \geq 0$ define
$\Lambda_\mu:=\big\{\lambda \in (0,\infty)\,:\, (\lambda,\mu)\in S \big\}$ and define
$$
\lambda_*(\mu):=\left\{ \begin{aligned}
\sup \Lambda_\mu &, \quad \hbox{if }\Lambda_{\mu}\neq \phi\\
0\quad &,\quad \hbox{if } \Lambda_{\mu}=\phi.
\end{aligned}
\right.
$$

Observe that $\lambda_*(\mu)=0$ if and only if $\Lambda_{\mu}=\phi$. From Proposition \ref{slnminproblemJmulambda} (or Theorem 1.3 part III in \cite{dosSantos2009-2}), $\lambda_*(0)>0$. 

\vskip 3pt 
First, we prove part {\it (i)}. Let $\overline{\mu} \geq 0$ and assume that $\lambda_{*}(\overline{\mu})\in (0,\infty]$. For any $\lambda\in (0, \lambda_*(\overline{\mu}))$ we may select $\overline{\lambda}$ such that 
$$
0<\lambda<\overline{\lambda}<\lambda_*(\overline{\mu}) \quad \hbox{and} \quad (\overline{\lambda},\overline{\mu})\in S.
$$

From Lemmas \ref{lemma:theo1.4_II}, \ref{lemma:theo1.4_IV},   \ref{lemma:theo1.4_III}, \ref{lemma:theo1.4_V} and \ref{lemma:theo1.4_VI} for any $\mu\in [0,\overline{\mu}]$, \eqref{HS1} has two nontrivial solutions. Besides, since $\lambda$ and $\mu$ are arbitrary, we conclude also  that $(0,\lambda_*(\overline{\mu}))\times [0,\overline{\mu}]\subset S$. This concludes the proof of part {\it i.}.

Moreover, the above considerations together with Lemma \ref{lemma:theo1.4_VII} show that given any $(\overline{\lambda},\overline{\mu})\in S$, then $(0,\overline{\lambda}]\times [0,\overline{\mu}]\subset S$. This shows that $S$ is path-connected.\\

Next, let $\mu_1$ and $\mu_2$ with $0\leq \mu_1<\mu_2<\infty$. We show that $\lambda_*(\mu_1)\geq \lambda_*(\mu_2)$. Arguing by contradiction, assume that $\lambda_*(\mu_2)>\lambda_*(\mu_1)$. Thus, for some $\lambda\in (\lambda_*(\mu_1), \lambda_*(\mu_2))$, we find that $(0,\lambda]\times [0,\mu_1]\subset S$, which contradicts the definition of $\lambda_*(\mu_1)$. This proves the claim and shows that $\lambda_*$ is decreasing.

Part {\it (ii)} in Theorem \ref{theo:bifur_diagram} follows directly from the definition of $\lambda_*(\mu)$ and from Lemma \ref{lemma:theo1.4_VII}. Similarly, part {\it (iii)} follows directly from the definition of $\lambda_*(\mu)$. Part {\it (v)} of Theorem \ref{theo:bifur_diagram} follows directly, for $\mu>0$, from the definition of $\lambda_*(\mu)$ and \eqref{ineq:mulambdahyperb}. The case $\mu=0$ follows from Theorem 1.3 part V in \cite{dosSantos2009-2}.\\
{To prove {\it iv.}, and in view of the non-increasing character of $\lambda_*$, it is enough to prove that on any open interval of the set $\{\mu\in (0,\infty)\,/\,\lambda_*(\mu)\in (0,\infty)\}$, $\lambda_*$ is not constant. We start by recalling the inequality 
\begin{equation}\label{ineq:concv-convex}
|x^{m}-y^{m}|\leq C_m|x-y|^m
\end{equation}
for $m\in (0,\infty)$ and $x,y\geq 0$.} 

{Arguing by contradiction let $0<\mu_1<\mu_2$ be such that for any $\mu\in [\mu_1,\mu_2]$, $\lambda_*(\mu)=\lambda_*\in (0,\infty)$. Let also $\mu\in (\mu_1,\mu_2)$ be arbitrary and consider $u_*,v_*\in C^2_0(\overline{\Omega})$ such that $(u_*,v_*)$ solve the system \eqref{HS1} for the pair \red{$(\lambda_*,\mu)$}. Consider also $\varphi\in C^2_0(\overline{\Omega})$ such that $-\Delta \varphi=v_*^r$ and $\varphi>0$ in $\Omega$.}  

{Let $\epsilon, \delta>0$ be arbitrary and such that $(\mu-\delta,\mu+\delta )\subset [\mu_1,\mu_2]$, $\epsilon/\delta^{\frac{1}{s}}$ is sufficiently small, and \begin{equation}\label{ineq:supersln_u}
u_*\geq \frac{\epsilon}{\delta^{\frac{1}{s}}} \varphi\Big[\mu C_s + C_q\epsilon^{q-s}\|\varphi\|^{q-s}_{L^{\infty}(\Omega)}\Big]^{\frac{1}{s}} \quad \hbox{in} \quad \Omega,
\end{equation}
where $C_s$ and $C_q$ are the constants from the corresponding  inequalities in \eqref{ineq:concv-convex} for $m=s$ and $m=q$, respectively. Observe also that the fact that $q>s$ and Hopf's Lemma applied to $u_*$ and $\varphi$ yield an open set in $(0,\infty)\times (0,\infty)$ of pairs $(\delta,\epsilon)$ satisfying the previous assumptions and this open set depends only on the choice of on $u_*$ and $v_*$.}  

{Set $u=u_*+\epsilon \varphi$ and $v=v_*$. We show next that $(u,v)$ is a supersolution of the system \eqref{HS1} for the pair of parameters $(\mu-\delta,\lambda_*+\epsilon)$. We calculate first,
$$
\begin{aligned}
    -\Delta u  =&-\Delta u_* - \epsilon\Delta \varphi\\
    =& \lambda_*+v_*^r +v_*^p + \epsilon v_*^r\\
    \geq&(\lambda_* + \epsilon)v^r +v^p\quad \hbox{in} \quad \Omega.
\end{aligned}
$$}

{On the other hand,
$$
\begin{aligned}
    -\Delta v =&-\Delta v_*\\
    =& \mu u_*^s +u_*^q\\
    =&(\mu -\delta )(u_* + \epsilon\varphi)^s + (u_* + \epsilon\varphi)^q+\\
    &\underbrace{+\delta(u_*+\epsilon\varphi)^s+\mu[u_*^s - (u_*+\epsilon\varphi)^s]+[u_*^q - (u_*+\epsilon\varphi)^q]}_{A}.
\end{aligned}
$$}

{The claim will follow once we show that $A\geq 0$. To prove this, we use \eqref{ineq:concv-convex} for $m=s$ and $m=q$ to find that
$$
\begin{aligned}
    A \geq & \delta(u_*+\epsilon\varphi)^s+\mu[u_*^s - (u_*+\epsilon\varphi)^s]+[u_*^q - (u_*+\epsilon\varphi)^q]\\
    \geq & \delta(u_*+\epsilon \varphi)^s -\mu C_s \epsilon^s\varphi^s - C_q\epsilon^q\varphi^q\\
    \geq & \delta u_*^s -\mu C_s \epsilon^s\varphi^s - C_q\epsilon^q\varphi^q\\
    \geq & \delta u_*^s -\epsilon^s \varphi^s\Big[\mu C_s + C_q\epsilon^{q-s}\|\varphi\|^{q-s}_{L^{\infty}(\Omega)}\Big].
\end{aligned}
$$}

{In view of \eqref{ineq:supersln_u}, $A\geq 0$ and this proves the claim.\\}

{We may now use the $(u,v)$ as supersolution for the pair \red{$(\lambda_*+\epsilon,\mu-\delta)$}. As subsolution we may use the pair $(\underline{u},\underline{v})$ described in \eqref{eqn:aux_system_sublinear} for the parameter $\lambda+\epsilon$. From Lemma \ref{lemma:theo1.4_I}, $\underline{u}\leq u$ and $\underline{v}\leq v$ in $\Omega$. Thus, proceeding as we did in Lemmas \ref{lemma:theo1.4_II}, \ref{lemma:theo1.4_III}, \ref{lemma:theo1.4_IV} and \ref{lemma:theo1.4_V}, or using Lemma 2.1 in (\cite{AGUDELORUFVELEZ2023}, pg 2908), we obtain the existence of a nontrivial solution of the system \eqref{HS1} for the pair \red{$(\lambda_*+\epsilon,\mu-\delta)$}, i.e., \red{$(\lambda_*+\epsilon,\mu-\delta)\in S$}. It also follows that $\lambda_*=\lambda_*(\mu-\delta)\geq\lambda_*+\epsilon$, which is a contradiction and this proves the claim. We have then proven that $\lambda_*$ is strictly decreasing.}

{A remark is in order. In the previous argument we used the assumption $q>s$, but the hypothesis $qr<1$ is not essential, once the existence of $u_*$ and $v_*$ has been stablished. Hence, a symmetric argument can be developed when switching the two equations.\\}

{Finally, we show the continuity of $\lambda_*$. The proof of the continuity of $\lambda_*$ from the left follows closely the lines of the proof of Lemma 7.7 and we leave the details to the reader. It is less obvious to show that $\lambda_*$ is continuous from the right in $(0,\infty)$. The proof of this statement is analogous to the proof of part {\it iv.} above. Let $\{\mu_n\}_n\subset (0,\infty)$ and $\mu>0$ be such  that 
$$
\hbox{for any }n\in \mathbb{N},\quad \mu\leq \mu_{n+1}< \mu_n \quad \hbox{and}\quad \mu_n\to \mu.
$$}  

{For $n\in \mathbb{N}$, set $\lambda_n:=\lambda_*(\mu_n)$ and $\lambda_*=\lambda_*(\mu)$. Observe that $\lambda_n< \lambda_{n+1}\leq \lambda_*$ so we may write $\lambda=\lim \limits_{n\to \infty}\lambda_n$. Clearly, $\lambda\leq \lambda_*$.}

{Now we prove that $\lambda=\lambda_*$. Using the above notations, let $u_*,v_*\in C^2_0(\overline{\Omega})$ be such that the pair $(u_*,v_*)$ is a solution of system \eqref{HS1} for the pair $(\mu,\lambda_*)$. Fix $\psi\in C^2_0(\overline{\Omega})$ such that $-\Delta \psi=u_*^s$ and $\psi>0$ in $\Omega$. Proceeding as we did to prove {\it iv.}, we find $\epsilon_0>0$ such that for any $\epsilon\in (0,\epsilon_0]$, there exists $\delta_{\epsilon}>0$ small such that for any $\delta\in (0,\delta_{\epsilon}]$, $u=u_*$ and $v=v_*+\delta \psi$ are such that $(u,v)$ is a super solution of the system \eqref{HS1}. This yields \red{$(\lambda_*-\epsilon,\mu+\delta)\in S$}.}

{Next, let $\epsilon\in (0,\epsilon_0)$ be arbitrary. If $n\in \mathbb{N}$ is large enough, then $\delta_n=\mu_n-\mu$ is such that $\delta_n\in (0,\delta_0)$ and thus \red{$(\lambda_*-\epsilon,\mu_n)=(\lambda_*-\epsilon,\mu+\delta_n)\in S$}. This in turn, implies that $\lambda_*(\mu_n)\geq \lambda_*-\epsilon$ and taking the limit as $n\to \infty$ we find that $\lambda\geq \lambda_*-\epsilon$. Since $\epsilon$ is arbitrary, we conclude that $\lambda=\lambda_*$ and this concludes the proof of the continuity of $\lambda_*$ from the right and concludes also the proof of the Theorem.} 
\QEDA

\medskip
{\underline{\it Proof of Theorem \ref{theo:Esp-caseMainTheo}}.
We only show that $\lambda_*(\mu)\to 0^+$ when $\mu\to \overline{\mu}$, since all the other properties of $\lambda_*$ follow directly from Theorem \ref{theo:bifur_diagram}. The fact that $\overline{\mu}\in (0,\infty)$ follows from Theorems 1.2 and 1.3 in \cite{AGUDELORUFVELEZ2023}. In any case, the following argument is based on the fact that $\overline{\mu}\in (0,+\infty)$, and that for $\lambda_*=\lambda_*(\overline{\mu})$,
$$
\overline{\mu}=\sup\{\mu\in (0,\infty)\,/\, (\lambda_*,\mu)\in S\}.
$$}

{Thus, let $\overline{\mu}\in (0,+\infty)$ and assume by contradiction that $\lambda_*=\lambda_*(\overline{\mu})\in (0,\infty)$.} 

{Using the hypotheses $qr<1$ and $ps<1$, symmetrical arguments to those developed throughout this work, allows us to consider for $\lambda \geq 0$, 
$M_{\lambda}:=\big\{\mu \in (0,\infty)\,:\, \red{(\lambda,\mu)}\in S \big\}$ and define
$$
\mu_*(\lambda):=\left\{ \begin{aligned}
\sup M_\lambda &, \quad \hbox{if }M_\lambda\neq \phi\\
0\quad &,\quad \hbox{if } M_\lambda=\phi.
\end{aligned}
\right.
$$}

{It holds that \red{$(\lambda_*,\mu_*(\lambda_*))\in S$} and $\overline{\mu}\leq \mu_*$. Hence, $\overline{\mu}=\mu_*(\lambda_*)$. Also, the fact that $\lambda_*>0$ implies that $\mu_*$ is constant on the interval $(0,\lambda_*]$ for $\lambda$. Using a symmetric argument as the one developed in the proof of Theorem \ref{theo:bifur_diagram}, to prove that $\lambda_*$ can no be constants on open intervals of $\mu$, yields a contradiction. We finish by commenting that from the definition of the sets $\Lambda_\mu$ and $M_{\lambda}$ it is verified that $\mu_*$ is the inverse function of $\lambda_*$. This completes the proof of the Theorem. 
\QEDA}

\bigskip
{\bf Acknowledgments.} The authors would like to express their gratitude to Professor Thomas Bartsch for enlightening conversations with the first author. These conversations motivated some of the developments herein presented. \\

The research of the first  author was supported by the Grant
CR 22-18261S of the Grant Agency of the Czech Republic. The research of the
third author was supported by Universidad Nacional de Colombia, Sede Medell\'{i}n,
Facultad de Ciencias, Departamento de Matem\'aticas, Grupo de Investigaci\'on en Matem\'aticas de la Universidad Nacional de Colombia Sede Medell\'in, Proyecto ``An\'alisis no lineal aplicado a problems mixtos en ecuaciones diferenciales parciales'', Fondo de Investigaci\'on de la Facultad de Ciencias, Hermes project code 60827.

\section{Appendix}

In this section we provide detailed proofs of some technical results used to set up the functional analytic framework and the variational structure of the BVP \eqref{4thorderBVP}-\eqref{eqn:navier_bd_cond}.

\vskip 3pt

{\it Proof of Lemma \ref{asymptoticpsizeroinfinity}.}
Clearly, given $\mu \in [0,\infty)$, $\psi(\mu,\cdot)$ is odd and strictly increasing with $\psi(\mu,0)=0$ and hence the inverse function theorem yields that $\psi(\mu,\cdot)\in C^1(\R-\{0\})$.

\vskip 3pt
For any $\theta \neq 0$, set $
\zeta = \psi(\mu, \theta)$, i.e., $\theta=g(\mu,\zeta)$. If either $\mu>0$ or $q<1$, we compute
$$
\partial_{\theta}\psi(\mu, 0)=\lim \limits_{\theta \to 0}\frac{\psi(\mu, \theta) - \psi(\mu,0)}{\theta} = \lim \limits_{\zeta \to 0}\frac{\zeta}{g(\mu,\zeta) - g(\mu,0)} = 0
$$ 
and notice that
$$
\partial_{\theta}\psi(\mu, \theta)= \left[\partial_{ \zeta}g(\mu, \psi(\mu,\theta))\right]^{-1} \quad \hbox{for} \quad \theta \neq 0.
$$

Therefore, $\psi(\mu, \cdot)\in C^{1}(\R)$, proving $(i)$ and $(ii)$.

To prove $(iii)$, it suffices to notice that
$$
\zeta g(\mu,\zeta)= \mu |\zeta|^{s+1}+ |\zeta|^{q+1}>0 \quad \hbox{for} \quad \zeta \neq 0.
$$

To prove $(iv)$, {for $\mu>0$} we first compute
$$
\begin{aligned}
\lim \limits_{|\theta|\to 0} \frac{\psi(\mu,\theta)}{|\theta|^{\frac{1}{s}-1}\theta}=& \lim \limits_{|\zeta|\to 0} \frac{\zeta}{|g(\mu,\zeta)|^{\frac{1}{s}-1}g(\mu,\zeta)}\\
=&\lim \limits_{|\zeta|\to 0} \frac{|\zeta|}{\left(\mu |\zeta|^{s}+ |\zeta|^{q}\right)^{\frac{1}{s}}}\\
=&\mu^{-\frac{1}{s}}.
\end{aligned}
$$

Proceeding similarly, using that $q>s$, we prove the {first} limit in $(iv)$.

\vskip 3pt
Finally, noticing that $\partial_{\zeta} g(\mu, \cdot)>0$ for $\zeta \neq 0$, a direct application of the {\it Implicit Function Theorem} to the function 
$$
\R\times \R \times \R \ni (\mu,\theta,\zeta) \mapsto g(\mu,\zeta) - \theta
$$ 
and its zero level set yields that $\psi\in C^{1}\left({(}0,\infty)\times \R\right)$. 
\QEDB

\vskip 3pt
{\it Proof of Lemma \ref{asymptoticsPSI}.} From Lemma \ref{asymptoticpsizeroinfinity} we readily verify $(i), (ii)$ and $(iii)$. Next, we prove  $(iv)$ and $(v)$. Write $\theta=\theta_++\theta_-$, where 
$$
\theta_+:=\max(0,\theta) \quad \hbox{and} \quad \theta_-:=\min(0,\theta).
$$

{Using that $\Psi$ is an even function, we find that $\Psi(\mu,|\theta|)=\Psi(\mu,\theta)$ for any $\theta \in \R$, thus proving $(iv)$.}

\vskip 3pt
Setting $\zeta =\psi(\mu,\theta)$, we find from the {\it Change of Variables Theorem} and integration by parts that
$$
\begin{aligned}
\Psi(\mu,\theta)=& \int^{\theta}_0 \psi(\mu,\hat{\theta})d \hat{\theta}= \int_{0}^{\zeta} \hat{\zeta} \partial_{\hat{\zeta}}g(\mu,\hat{\zeta})d\hat{\zeta}\\
=& \zeta g(\mu,\zeta) - \int_{0}^{\zeta}g(\mu,\hat{\zeta})d \hat{\zeta}.
\end{aligned}
$$

Therefore, for any $\theta \in \R$,
\begin{equation}\label{relatPsiG}
\Psi(\mu,\theta) + G(\mu, \psi(\mu,\theta)) =\psi(\mu,\theta)\theta.
\end{equation}

Using $(iv)$ from Lemma \ref{asymptoticpsizeroinfinity}, the identity \eqref{relatPsiG} and that $0<s<q$, {we have for $\mu>0$,}
$$
\begin{aligned}
\lim \limits_{|\theta| \to 0} \frac{\Psi(\mu, \theta)}{|\theta|^{\frac{s+1}{s}}}=& \lim \limits_{|\theta| \to 0} \frac{\psi(\mu,\theta)\theta}{|\theta|^{\frac{s+1}{s}}} - \lim \limits_{|\theta| \to 0} \frac{G(\mu,\psi(\mu,\theta))}{|\theta|^{\frac{s+1}{s}}}\\
=& \mu^{-\frac{1}{s}} - \lim \limits_{|\zeta| \to 0} \frac{\frac{\mu}{s+1}|\zeta|^{s+1} + \frac{1}{q+1}|\zeta|^{q+1}}{\left(\mu|\zeta|^s + |\zeta|^q\right)^{\frac{s+1}{s}}}\\
=&\frac{s}{s+1}\mu^{-\frac{1}{s}}.
\end{aligned}
$$

Proceeding in the same fashion, we compute 
$$
\begin{aligned}
\lim \limits_{|\theta| \to \infty} \frac{\Psi(\mu, \theta)}{|\theta|^{\frac{q+1}{q}}}=& \lim \limits_{|\theta| \to \infty} \frac{\psi(\mu,\theta)\theta}{|\theta|^{\frac{q+1}{q}}} - \lim \limits_{|\theta| \to \infty} \frac{G(\mu,\psi(\mu,\theta))}{|\theta|^{\frac{q+1}{q}}}\\
=& 1 - \lim \limits_{|\zeta| \to \infty} \frac{\frac{\mu}{s+1}|\zeta|^{s+1} + \frac{1}{q+1}|\zeta|^{q+1}}{\left(\mu|\zeta|^s + |\zeta|^q\right)^{\frac{q+1}{q}}}\\
=&\frac{q}{q+1}
\end{aligned}
$$
and this completes the proof of $(v)$.

\vskip 3pt
A direct calculation using \eqref{relatPsiG}, shows that
$$
\frac{\partial}{\partial \theta}\left[\frac{\Psi(\mu,\theta)}{\theta}\right]= \frac{\theta \psi(\mu,\theta)- \Psi(\mu,\theta)}{\theta^2}= \frac{G(\mu,\psi(\mu,\theta))}{\theta^2} >0
$$
for any $\theta \neq 0$, which proves $(vi)$.

\QEDB

\vskip 3pt
{\it  Proof of Lemma \ref{inequalitiespsi}.} Observe that
$$
\frac{1}{q+1}\zeta g(\mu,\zeta)  \leq G(\mu,\zeta) \leq \frac{1}{s+1}\zeta g(\mu,\zeta) \quad \hbox{for} \quad \zeta \in \R,
$$
which together with \eqref{relatPsiG}, taking $\zeta=\psi(\mu,\theta)$, yields \eqref{ComparisonInequality}.

Next, setting $\tilde{C}_{q,s}:= 1 + \frac{q(q-s)}{s(1-s)}$ and $\tilde{c}_{q,s}:= 1 + \frac{s(1-s)}{q(q-s)}$, we find that
$$
|\zeta|^{q+1}\leq g(\mu, \zeta)\zeta \leq \left\{
\begin{aligned}
\tilde{C}_{q,s}|\zeta|^{q+1},  &\quad \hbox{for} & |\zeta| \geq \zeta_\mu\\
\mu \tilde{c}_{q,s}|\zeta|^{s+1}, & \quad \hbox{for} & |\zeta| \leq \zeta_\mu.
\end{aligned}
\right.
$$

Thus, \eqref{growthpsi} follows with $\hat{C}_{q,s}:=\tilde{C}_{q,s}^{-\frac{1}{q}}$ and $\hat{c}_{q,s}:= \tilde{c}_{q,s}^{-\frac{1}{s}}$. 

\QEDB

\vskip 3pt
{\it Proof of Lemma \ref{technicalPSineq}.} Set $\zeta_i = \psi(\mu,\theta_i)$ for $i=1,2$ and assume w.l.o.g. that $\zeta_1 > \zeta_2$. Assume first that {$q> 1>s$}, from Young's inequality
\begin{equation}\label{youngineq1}
|\zeta_1|^{q-1}|\zeta_2|^{1-s}\leq \frac{q-1}{q-s}|\zeta_1|^{q-s} + \frac{1-s}{q-s}|\zeta_2|^{q-s}
\end{equation}
with a similar estimate for the term $|\zeta_1|^{1-s}|\zeta_2|^{q-1}$. Thus, we estimate  
$$
\begin{aligned}
|g(\mu,\zeta_1) - g(\mu,\zeta_2)|\leq & \mu \Bigl||\zeta_1|^{s-1}\zeta_1 - |\zeta_2|^{s-1}\zeta_2\Bigr|+\Bigl||\zeta_1|^{q-1}\zeta_1 - |\zeta_2|^{q-1}\zeta_2\Bigr|\\
\leq & C_s\mu|\zeta_1 - \zeta_2|^{s} + q\max(|\zeta_1|^{q-1}; |\zeta_2|^{q-1})|\zeta_1 - \zeta_2|
\end{aligned}
$$ 
and using \eqref{youngineq1},
\begin{equation}\label{ineq:hold_g_s}
|g(\mu,\zeta_1)-g(\mu,\zeta_2)|\leq C_{s,q} \left(\mu + |\zeta_1|^{q-s}+ |\zeta_2|^{q-s}\right)|\zeta_1 -\zeta_2|^s.
\end{equation}

Assume now that $0<s<q\leq 1$. It is readily checked that also in this case \eqref{ineq:hold_g_s} holds true.

\vskip 3pt
Since $g$ is strictly increasing, we conclude from the latter inequality that
$$
\begin{aligned}
(\zeta_1 - \zeta_2)(g(\mu,\zeta_1)-g(\mu,\zeta_2)) =& |\zeta_1 - \zeta_2||g(\mu,\zeta_1)-g(\mu,\zeta_2)|\\
\geq &c_{s,q}\frac{|g(\mu,\zeta_1) - g(\mu,\zeta_2)|^{\frac{s+1}{s}}}{\left(\mu + |\zeta_1|^{q-s}+ |\zeta_2|^{q-s}\right)^{\frac{1}{s}}}\\
\geq &\tilde{c}_{s,q}\frac{|g(\mu,\zeta_1) - g(\mu,\zeta_2)|^{\frac{s+1}{s}}}{\left(\mu^{\frac{1}{s}} + |\zeta_1|^{\frac{q-s}{s}}+ |\zeta_2|^{\frac{q-s}{s}}\right)}
\end{aligned}
$$
and \eqref{monotonicitypsi} follows. This completes the proof.

\QEDB

\bigskip


\begin{thebibliography}{99}

\bibitem{AgudeloKudlacHolubova2024} {\sc Agudelo O., Kudl\'a\v{c} M. Holubov\'{a} G.}Variational and numerical aspects of a system of ODEs with concave-convex nonlinerities. arXiv:2408.10630v1 [math.FA] 2024. 

\bibitem{AGUDELORUFVELEZ2023} {\sc Agudelo O., Ruf B., V\'{e}lez C.} On a Hamiltonian elliptic system with concave and convex nonlinearities. {\it DCDS Serie S.} {\bf 2023}. Doi: 10.3934/dcdss.2023103
 


\bibitem{AMBROSSETIBREZISCERAMI1994}
{\sc Ambrosetti A.; Brezis  H.; Cerami G.} Combined effects of concave and convex nonlinearities in some elliptic problems. {\it J. Funct. Anal.} {\bf 122} (1994), no. 2, 519-543.


\bibitem{AMBROSETTIRABINOWITZ1973} {\sc Ambrosetti A.; Rabinowitz P.H.}
Dual variational methods in critical point theory and applications.
{\it J. Funct. Anal.} {\bf 14} (1973), 349-381. 



\bibitem{BrezisBook2011}  {\sc Br\'ezis H.} Functional analysis, Sobolev spaces and partial differential equations. {\it Universitext}. Springer, New York, 2011. xiv+599 pp. 


\bibitem{BREZISNIRENBERG1983} {\sc Br\'{e}zis H.; Nirenberg L.} Positive solutions of nonlinear elliptic equations involving critical Sobolev exponents.  {\it Comm. Pure Appl. Math.} 36 (1983), no. 4, 437-477. 



\bibitem{BREZISNIRENBERG1993}  {\sc Br\'{e}zis H.; Nirenberg L.} $H^1$ versus $C^1$ local minimizers. {\it C. R. Acad. Sci. Paris S\'{e}r. I Math.} 317 (1993), no. {\bf 5}, 465-472.


\bibitem{BREZISCASENAVEMARTEL1996} {\sc Br\'ezis H., Cazenave T., Martel Y.; Ramiandrisoa A.} Blow up for $u_t -
\Delta u = g(u)$ revisited. {\it Adv. Differential Equations} 1 (1996), no. {\bf 1}, 73-90.

\bibitem{IturrUbillBrock2008} {\sc Brock F.; Iturriaga L.; Ubilla P.} A multiplicity result for the p-Laplacian involving a parameter.{\it Ann. Henri Poincaré} 9 (2008), no. {\bf 7}, 1371–1386.


\bibitem{ClarkeEkeland1980} {\sc Clarke F.H.; Ekeland I.} Hamiltonian trajectories having prescribed minimal period. {\it Comm. Pure Appl. Math.} {\bf 33}, 103-116, 1980.

\bibitem{DrabekMilota} {\sc Drabek P.; Milota J.} Methods of Nonlinear Analysis: Applications to Differential Equations. Second Edition. {\it (Birkhauser Advanced Texts}. Basler Lehrbacher. Birkhäuser/Springer Basel AG, Basel, 2013.

\bibitem{DeFigueiredo1987} {\sc de Figueiredo, D.G.} On the existence of multiple ordered solutions of nonlinear eigenvalue problems. {\it Nonlinear Anal.} 11 (1987), no. {\bf 4}, 481–492. 

\bibitem{DeFigueiredo1996} {\sc de Figueiredo, D.G.} Semilinear elliptic systems: a survey of superlinear problems. {\it Resenhas} {\bf 2} (1996), no. 4, 373–391.


\bibitem{felmerdefigueiredo1994} {\sc de Figueiredo D.G. and Felmer P.L.} On superquadratic elliptic systems, {\it Trans. Amer. Math. Soc.} {\bf 343} (1994) p. 99-116.



\bibitem{DiFrattaFiorenza2020} {\sc di Fratta G.; Fiorenza A.} A short proof of local regularity of distributional solutions of Poisson's equation. {\it Proc. Amer. Math. Soc}. {\bf 148} (2020), no. 5, 2143-2148.




\bibitem{GilbargTrudinger} {\sc Gilbarg D.; Trudinger N.S.} Elliptic partial differential equations of second order. Second edition. {\it Fundamental Principles of Mathematical Sciences}, {\bf 224}. Springer-Verlag, Berlin, 1983.


\bibitem{Guimaraes-dosSantos2023}  {\sc Guimar\~{a}es A.; dos Santos E.M.} On Hamiltonian systems with critical Sobolev exponents. {\it Journal of  Differential Equations,} {\bf 360} (2023), 314-346.


\bibitem{HulshofVanderVorst1993} {\sc Hulshof J.; Van der Vorst  R.} Differential systems with strongly indefinite variational structure. {\it J. Funct. Anal.} 114 (1993), no. {\bf 1}, 32-58.

\bibitem{HulshofVanderVorst1996} {\sc Hulshof J.; Van der Vorst  R.}  Asymptotic behaviour of ground states. Proc. Amer. Math. Soc. {\bf 124} (1996), no. 8, 2423–2431. 



\bibitem{HulshofMitidieriVanderVorst1998} {\sc Hulshof J.; Mitidieri E.; Van der Vorst  R.}  Strongly indefinite systems with critical Sobolev exponents. Trans. Amer. Math. Soc. {\bf 350} (1998), no. {6}, 2349–2365. 




%

\bibitem{Melo-dosSantos2015}
{\sc Melo J.L.F., Dos Santos E.M.}. Critical and noncritical regions on the critical hyperbola. In: Nolasco de Carvalho, A., Ruf, B., Moreira dos Santos, E., Gossez, JP., Monari Soares, S., Cazenave, T. (eds) Contributions to Nonlinear Elliptic Equations and Systems. Progress in Nonlinear Differential Equations and Their Applications, vol 86. Birkh\"{a}user, Cham. https://doi.org/10.1007/978-3-319-19902-3-21



\bibitem{Mitidieri1993} {\sc Mitidieri E.} A Rellich type identity and applications.
Comm. Partial Differential Equations {\bf 18} (1993), no. 1-2, 125-151. 


\bibitem{MotrMotrPapa2014}  {\sc Motreanu D.; Motreanu  V.V.; Papageorgiou N.} Topological and variational methods with applications to nonlinear boundary value problems Springer, New York,  2014, xii+459 pp.ISBN: 978-1-4614-9322-8; 978-1-4614-9323-5.

%



\bibitem{dosSantos2009} {\sc dos Santos E.M.} On the existence of positive solutions for a nonhomogeneous elliptic system. Port. Math. {\bf 66} (2009), no. 3, 347-371.


\bibitem{dosSantos2009-2}  {\sc dos Santos E.M.} On a fourth-order quasilinear elliptic equation of concave-convex type. {\it Nonlinear Differ. Equ. Appl.} {\bf 16} (2009), 297-326.






\bibitem{Pohozaev1965} {\sc Poho\v{z}aev S.I.} On the eigenfunctions of the equation $\Delta u+\lambda f(u)=0$. (Russian)
Dokl. Akad. Nauk SSSR {\bf 165} 1965 36-39.

\bibitem{PUCCISERRIN1984}  {\sc Pucci P.; Serrin J.}Extensions of the mountain pass theorem
 Funct. Anal. 59 (1984), no. 2, 185–210.

\bibitem{StruweBook2008}  {\sc Struwe M.} Variational methods. Applications to nonlinear partial differential equations and Hamiltonian systems. Fourth edition. {\it A Series of Modern Surveys in Mathematics}, {\bf 34}. Springer-Verlag, Berlin, 2008.


\bibitem{VainbergBook} {\sc Vainberg M.M.} Variational methods for the study of nonlinear operators. San
Francisco, Calif.-London-Amsterdam 1964.


\bibitem{Vandervorst1992} {\sc Van der Vorst R.} Variational identities and applications to differential systems. Arch. Rational Mech. Anal. {\bf 116} (1992), no. 4, 375-398.
\end{thebibliography}
\end{document}